\documentclass[]{interact}

\usepackage{epstopdf}
\usepackage[caption=false]{subfig}

\usepackage[numbers,sort&compress]{natbib}
\bibpunct[, ]{[}{]}{,}{n}{,}{,}

\theoremstyle{plain}
\newtheorem{theorem}{Theorem}[section]
\newtheorem{lemma}[theorem]{Lemma}

\newtheorem{proposition}[theorem]{Proposition}

\theoremstyle{definition}

\newtheorem{example}[theorem]{Example}

\theoremstyle{remark}
\newtheorem{remark}{Remark}

\usepackage{mathrsfs}
\usepackage[latin1]{inputenc}
\usepackage[T1]{fontenc}  
\usepackage{xcolor}
\usepackage{accents}
\usepackage{import}
\usepackage{graphicx}
\usepackage{hyperref}
\usepackage{float}

\newcommand{\R}{\mathbb{R}}

\newcommand{\M}{\mathcal{M}}

\newcommand{\Q}{\mathbb{Q}}

\newcommand{\T}{\mathbb{T}}
\newcommand{\E}{\mathbb{E}}
\newcommand{\Pb}{\mathbb{P}}
\newcommand{\Qb}{\mathbb{Q}}

\newcommand{\cF}{\mathcal{F}}

\newcommand{\cH}{\mathcal{H}}

\newcommand{\cL}{\mathcal{L}}
\newcommand{\cOML}{\tilde{\mathcal{L}}}
\newcommand{\cV}{\mathcal{V}}

\newcommand{\cOM}{\mathcal{OM}}
\newcommand{\cFM}{\mathcal{FM}}

\newcommand{\partiel}[1]{\frac{\partial}{\partial {#1}}}

\newcommand{\SO}{\mathrm{SO}}

\DeclareMathOperator{\Log}{Log}
\DeclareMathOperator{\Exp}{Exp}
\DeclareMathOperator{\Vol}{Vol}
\DeclareMathOperator{\Cut}{Cut}
\DeclareMathOperator{\Hess}{Hess}

\usepackage{mathtools}
\DeclarePairedDelimiterX{\norm}[1]{\lVert}{\rVert}{#1}

\newcommand{\refthm}[1]{Theorem~\ref{#1}}
\newcommand{\refprop}[1]{Proposition~\ref{#1}}
\newcommand{\reflemma}[1]{Lemma~\ref{#1}}

\newcommand{\refsec}[1]{Section~\ref{#1}}

\newcommand{\reffig}[1]{Figure \ref{#1}}

\begin{document}


\title{Simulation of Conditioned Semimartingales on Riemannian Manifolds}

\author{
\name{Mathias H\o jgaard Jensen and Stefan Sommer\thanks{Contact: Stefan Sommer, sommer@di.ku.dk}}
\affil{Department of Computer Science, University of Copenhagen, Denmark}
}

\maketitle

\begin{abstract}
We present a scheme for simulating conditioned semimartingales taking values in Riemannian manifolds. Extending the guided bridge proposal approach used for simulating Euclidean bridges, the scheme replaces the drift of the conditioned process with an approximation in terms of a scaled radial vector field. This handles the fact that transition densities are generally intractable on geometric spaces. We prove the validity of the scheme by a change of measure argument, and we show how the resulting guided processes can be used in importance sampling and for approximating the density of the unconditioned process. The scheme is used for numerically simulating bridges on two- and three-dimensional manifolds, for approximating otherwise intractable transition densities, and for estimating the diffusion mean of sampled geometric data.
\end{abstract}

\section{Introduction}
Techniques for simulating Euclidean diffusion bridge processes, bridge sampling, have been studied in several cases over the last two decades, for instance \cite{clark1990simulation,delyon_simulation_2006,schauer2017guided,van2017bayesian}. Delyon \& Hu \cite{delyon_simulation_2006} introduced a scheme to simulate conditioned diffusions using a stochastic differential equation (SDE) that can be directly simulated and is absolutely continuous with respect to the desired distribution. The SDE interchanged a drift term depending on the possibly intractable transition density with a drift arising from a Brownian bridge going from $a$ to $b$ in time $T$ for $a,b \in \R^d$.

We here propose a scheme for simulating conditioned semimartingales on Riemannian manifolds. This generalizes the result of Delyon and Hu \cite{delyon_simulation_2006} to the setting of smooth connected Riemannian manifolds. We prove analytically that the result holds on a variety of manifolds, including the space of symmetric positive definite matrices considered by Bui et al. \cite{buiInferencePartiallyObserved2022}. More precisely, we show that by adding a drift term to the semimartingale, the manifold equivalent to the drift term introduced in \cite{delyon_simulation_2006} to approximate the gradient of the log-transition density, we obtain a process that converges to the desired endpoint. The process is absolutely continuous with respect to the targeted diffusion bridge distribution.

Bridge sampling is an essential part of likelihood and Bayesian inference for discretely observed stochastic processes. For example, bridge sampling finds applications in geometric statistics, medical imaging, and shape analysis. Recent papers have introduced algorithms to simulate stochastic bridges on specific manifolds \cite{buiInferenceRiemannianManifolds2022,buiInferencePartiallyObserved2022,arnaudon2020diffusion,jensen2019simulation,sommer2017bridge}, which indicate the necessity for bridge sampling algorithms on general Riemannian manifolds. In particular, the ability to estimate the transition density of a stochastic process enables likelihood-based and Bayesian approaches to geometric statistics.

Figure \ref{fig: sphere}(a) shows sample paths generated by the sampling scheme set up to approximate a conditioned Brownian motion on the sphere $\mathbb{S}^2$ starting at the north pole and conditioned to hit the south pole at $T=1$. The scaled squared radial vector field, $\tfrac{\nabla d(\cdot,v)^2}{2(T-t)}$ Figure \ref{fig: sphere}(b) acts as the guiding term forcing the process towards the target $v$.

\begin{figure}[h] 
    \centering
	\subfloat[Four simulated paths of a diffusion bridge process from the north pole (red point) to the south pole (black point)]{\includegraphics[width=.45\linewidth]{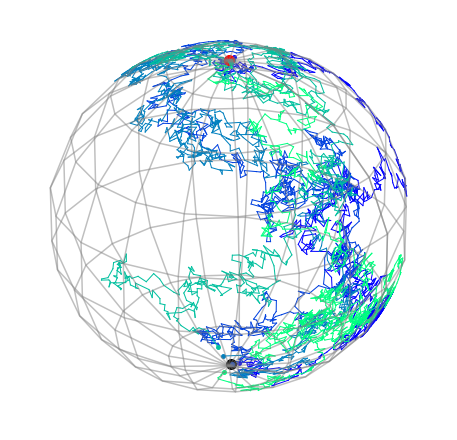}}
    \qquad
	\subfloat[Gradient vector field of the squared distance function on the $2$-sphere centered at the south pole.]{\includegraphics[width=.45\linewidth]{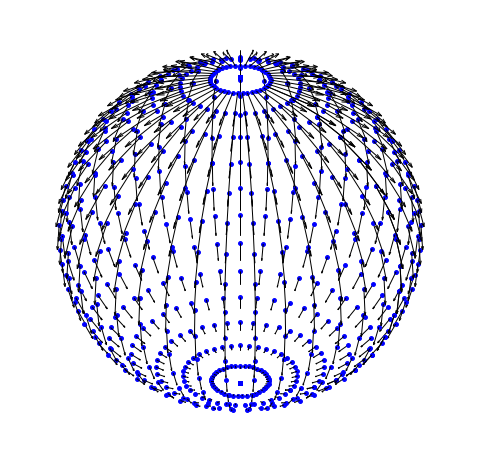}}
	\caption{ Realizations of the bridge simulation scheme in a case where the target point is in the cut locus of the starting point, and the guiding vector field that approximates the score.}\label{fig: sphere}
\end{figure}

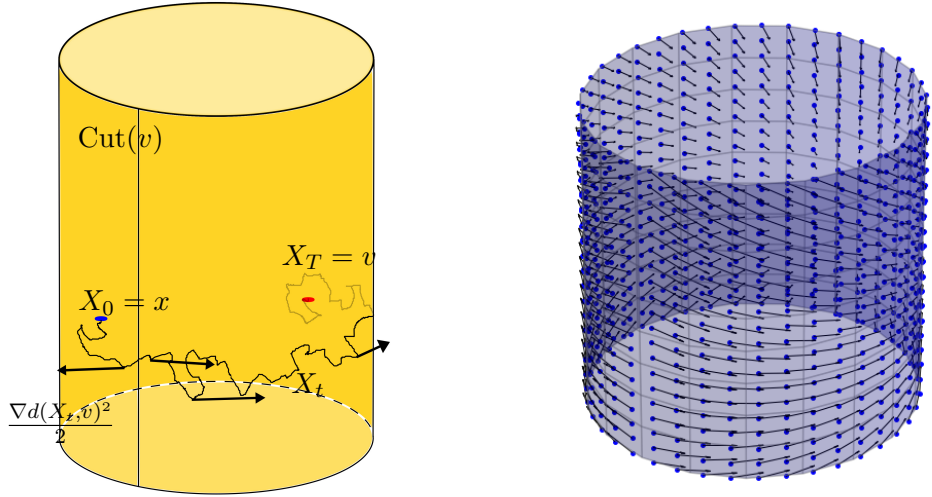
\begin{figure}[h] 
    \centering
	\subfloat[Illustration of a sample path of a radial bridge, $X_t$, from starting point $x$ to target $v$ with corresponding guiding drift indicated by arrows. The drift discontinuously changes sign when crossing the cut locus of the target (vertical line).]{\def\svgwidth{.45\columnwidth}
\begingroup%
  \makeatletter%
  \providecommand\color[2][]{%
    \errmessage{(Inkscape) Color is used for the text in Inkscape, but the package 'color.sty' is not loaded}%
    \renewcommand\color[2][]{}%
  }%
  \providecommand\transparent[1]{%
    \errmessage{(Inkscape) Transparency is used (non-zero) for the text in Inkscape, but the package 'transparent.sty' is not loaded}%
    \renewcommand\transparent[1]{}%
  }%
  \providecommand\rotatebox[2]{#2}%
  \newcommand*\fsize{\dimexpr\f@size pt\relax}%
  \newcommand*\lineheight[1]{\fontsize{\fsize}{#1\fsize}\selectfont}%
  \ifx\svgwidth\undefined%
    \setlength{\unitlength}{595.27559055bp}%
    \ifx\svgscale\undefined%
      \relax%
    \else%
      \setlength{\unitlength}{\unitlength * \real{\svgscale}}%
    \fi%
  \else%
    \setlength{\unitlength}{\svgwidth}%
  \fi%
  \global\let\svgwidth\undefined%
  \global\let\svgscale\undefined%
  \makeatother%
  \begin{picture}(1,1.41428571)%
    \lineheight{1}%
    \setlength\tabcolsep{0pt}%
    \put(0,0){\includegraphics[width=\unitlength]{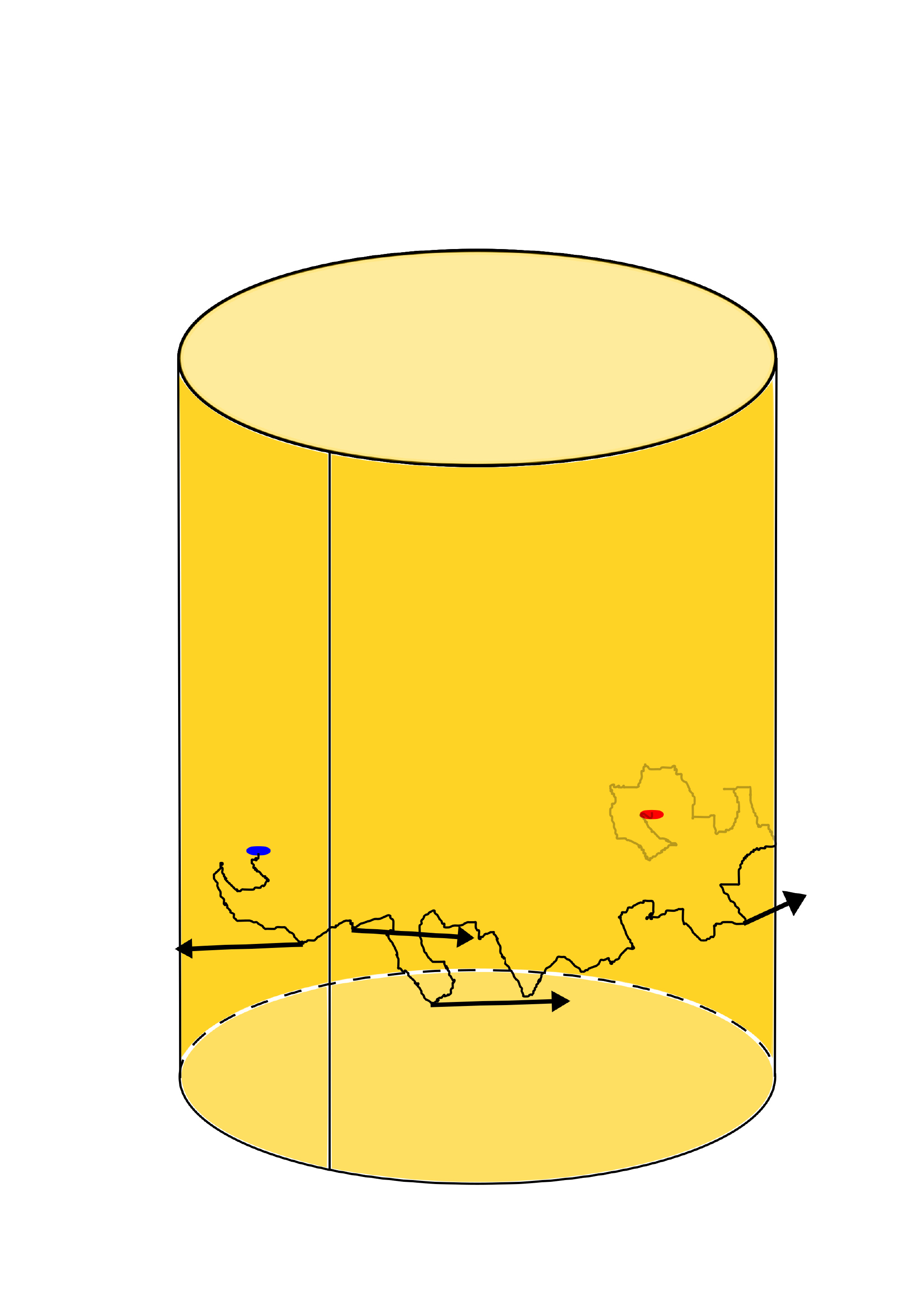}}%
    \put(0.08196267,0.25657879){\color[rgb]{0,0,0}\makebox(0,0)[lt]{\lineheight{1.25}\smash{\begin{tabular}[t]{l}$\frac{\nabla d(X_t,v)^2}{2}$\end{tabular}}}}%
    \put(0.23080609,0.89051951){\color[rgb]{0,0,0}\makebox(0,0)[lt]{\begin{minipage}{0.62006531\unitlength}\raggedright $\Cut(v)$\end{minipage}}}%
    \put(0.65078551,0.64341772){\color[rgb]{0,0,0}\makebox(0,0)[lt]{\begin{minipage}{0.26213001\unitlength}\raggedright $X_T=v$ \end{minipage}}}%
    \put(0.67076075,0.37971174){\color[rgb]{0,0,0}\makebox(0,0)[lt]{\begin{minipage}{0.17283429\unitlength}\raggedright $X_t$\end{minipage}}}%
    \put(0.2319933,0.54987339){\color[rgb]{0,0,0}\makebox(0,0)[lt]{\begin{minipage}{0.27127856\unitlength}\raggedright $X_0=x$\end{minipage}}}%
  \end{picture}%
\endgroup%
}
    \qquad
	\subfloat[Radial vector field on the cylinder, $X_t$, centered at the target $v$.]{\includegraphics[width=.45\linewidth]{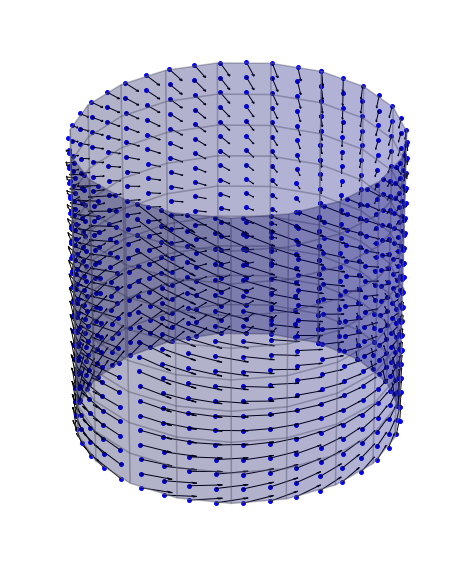}}
	\caption{Illustration of the effect of the cut locus on the radial bridge process. The guiding drift term, indicated by the tangent vectors along the sample path in (a) and the radial vector field (b) discontinuosly changes directions when crossing discontinuosly the cut locus.
  }\label{fig:cylinder cut locus}
\end{figure}

\subsection{Guided proposals}\label{section:delyon & hu}
Doob's $h$-transform is a classical way to show that the SDE for the conditioned process has a drift term that depends on the gradient of the logarithmic transition density, the score $\nabla_x \log p(t,x,v)$. Due to the generally intractable transition density, various methods to simulate conditioned processes have been introduced (e.g., \cite{clark1990simulation,delyon_simulation_2006,schauer2017guided,van2017bayesian}).

Delyon \& Hu \cite{delyon_simulation_2006} presented an algorithm to simulate the distribution of certain types of diffusions in $\R^d$ conditioned to hit a terminal point, at some fixed time $T > 0$.  Their main idea was to substitute the drift $\nabla_x \log p(t,x,v)$, which depends on the transition density, with the drift term appearing in the SDE for the generalized Brownian bridge. More precisely, they showed that the law of the process 
    \begin{equation*}
        dX_t = b(t,X_t)dt + \sigma(t,X_t)dW_t, \qquad X_0 = a,
    \end{equation*}
conditioned on $X_T = v$ is absolutely continuous with respect to the law of the process
	\begin{equation}\label{eq: proposed bridge Delyon and Hu}
		dY_t = b(t,Y_t)dt - \frac{Y_t-v}{T-t} + \sigma(t,Y_t) dW_t, \qquad Y_0 = a
	\end{equation}
under suitable conditions on $b$ and $\sigma$.
Moreover, the conditional expectation given $X_T=v$ satisfies
	\begin{equation}\label{delyon and hu result}
	 \E\left[f(X)|X_T=v\right] = C\E\left[f(Y)\varphi_T\right],
	\end{equation}
for a constant $C > 0$, where $\varphi_T$ denotes the likelihood ratio. We present a result that generalizes the SDE in \eqref{eq: proposed bridge Delyon and Hu} and essentially also equation \eqref{delyon and hu result} to Riemannian manifolds, with an explicit expression for the likelihood ratio $\varphi_T$. 

Switching from vector space to Riemannian manifolds, curvature removes the closed form solution of the Brownian bridge drift. Inspired by Delyon and Hu's construction and the notion of Fermi bridges \cite{thompson2015submanifold,thompson2018brownian}, we instead propose to use the gradient $\nabla_xd(\cdot,v)^2/2$ in the drift term in a non-Euclidean generalization of \eqref{eq: proposed bridge Delyon and Hu}. However, the existence of a non-trivial cut locus implies that this radial vector field is not continuous. Moreover, the $t\to 0$ convergence of the term to the gradient of the log-density is more intricate than the Euclidean situation.
As shown by Malliavin, Stroock \cite{malliavin1996short}, and Turetsky \cite{stroock1997short},
\begin{equation}\label{equation: convergence of grad log heat kernel}
  \lim_{t\to 0} t \nabla^m_x\log p_{\M}(t,x,v) = - \nabla^m_x d(x,v)^2/2, 		\qquad \text{for all } m \geq 0,
	\end{equation}
uniformly on compact subsets of $\M\backslash \Cut(v)$, where $\nabla^m$ is the $m$th covariant derivative and $p_{\M}$ the heat kernel or the transition density of a Brownian motion on a manifold $\M$. The behaviour of the right hand side of \eqref{equation: convergence of grad log heat kernel}, for $m=1$, is illustrated in Figure \ref{fig:cylinder cut locus}. The cut locus adds extra difficulty because the squared distance is not differentiable on the cut locus, and the convergence in \eqref{equation: convergence of grad log heat kernel} is only uniform away from the cut locus.

\subsection{Bridge sampling: Parameter, density, and metric estimation}
Bridge processes are useful in a range of statistical and applied mathematical problems. We here list a few examples:
Bridge sampling is essential for building data augmentation algorithms \cite{van2017bayesian}. Treating sample points as if coming from incomplete observations, bridge sampling algorithms provides a better understanding of the underlying distribution, thereby offering a method for parameter estimation. Delyon and Hu \cite{delyon_simulation_2006} presents a rather specific example of parameter estimation using bridge sampling.
Furthermore, bridge sampling techniques can be applied to estimate normalizing constants (see, e.g., \cite{gelman1998simulating}). 

Statistical models based on diffusion processes find applications in geometric statistics, see, e.g. \cite{pennec2019riemannian}. In particular, bridge sampling on manifolds yields estimates of the diffusion mean \cite{hansen2021diffusiongeometric,hansen2021diffusion}.  The diffusion mean relies on the transition density of a Brownian motion and is a generalization of the Fr\'echet mean, defined as the argument that minimizes the average of distances. In this case, bridge sampling approximates the transition density or data likelihood. We demonstrate in the numerical examples how bridge sampling can be used to approximate the transition density and to find diffusion means.

Estimation of transition densities by bridge sampling find uses in shape analysis \cite{arnaudon2017stochastic,arnaudon2019geometric,sommer2017bridge}, e.g. for estimating the underlying Riemannian metric given shape observations.

\subsection{Outline}
In \refsec{sec:background}, we review relevant existing work on manifold bridge processes and the frame bundle theory used to describe stochastic integration on manifolds. \refsec{sec:girsanov change of measure} describes the Radon-Nikodym derivative related to the change of measure and a Cameron-Martin-Girsanov change of measure result.
In \refsec{sec:2 main results - bride simulation on manifolds}, we present our two main results that generalize \cite[Theorem 5]{delyon_simulation_2006} to manifold-valued semimartingales.
\refsec{sec:radial part of semimartingale} is devoted to the radial process. We apply Barden and Le's result \cite[Theorem 3]{le1995ito} to the radial process and show the almost sure convergence of our guided process. 
We go on to rigorously prove the two propositions in \refsec{sec:proof of main result}, before ending with numerical examples in \refsec{sec:numerical experiments}. We visualize the result of the simulation scheme on two- and three-dimensional manifolds, and we use the approach to approximate the heat kernel on different 2-dimensional manifolds. Furthermore, we provide examples of how iterative bridge sampling can be used to estimate diffusion means.

\subsection{Notation and conventions}\label{sec:notation and convention}
Throughout the paper, we use the following notation and make the listed assumptions:
\begin{itemize}
 \item $(\Omega, \cF, (\cF_t)_{t \geq 0}, \Pb)$ is a complete filtered probability space satisfying the usual conditions, i.e., the filtration $(\cF_t)_{t \geq 0}$ is complete and right-continuous.
    \item $\M$ is a $d$-dimensional smooth manifold.
	\item $W_t$ a $d$-dimensional Euclidean Brownian motion.
	\item $Z_t$ an $d$-dimensional Euclidean semimartingale, a solution to the SDE $dZ_t = b(t,Z_t)dt + \sigma(t,Z_t)dW_t$, for suitably integrable $b$ and $\sigma$.
	\item $U_t$ a horizontal semimartingale in the frame bundle $\cFM$.
	\item $X_t$ a $\M$-valued semimartingale defined as the projection $X_t := \pi(U_t)$.
	\item $r_v(X_t) := d(X_t,v)$ is the Riemannian distance between $X_t$ and $v \in \M$.
	\item Both the diffusion field $\sigma_t(z) := \sigma(t,z)$ and its inverse $\sigma^{-1}_t(z)$ are bounded and differentiable	.
	\item The bracket process $[Z]$ is absolutely continuous as a random measure on $[0,\infty)$.
	\item The set $\{t \colon X_t \in \Cut(v)\}$ is a Lebesgue null-set.
	\item All SDEs admits strong solutions.
\end{itemize}
We will use $\tilde{\,}$ to denote lifts of functions on $M$ to fiber bundles, i.e., with projection $\pi$, $\tilde{f}:= f\circ \pi$.
As the distance function $d(\cdot,y)$ is non-differentiable on the cut locus, we take $\nabla_x d(x,y)$ to be the usual gradient away from cut locus and define it to be zero at the cut locus. 
We will repeatedly use the Einstein summation convention.

\section{Background}\label{sec:background}

\subsection{Manifold valued processes}\label{section: manifold valued diffusions}
We review the theory of frame bundles and manifold-valued diffusion processes. In particular, we describe the main concepts of differential geometry, which are needed to define the Eells-Elworthy-Malliavin construction of manifold-valued stochastic processes through horizontal lifts. Two common references for stochastic calculus on manifolds are Emery \cite{emery1989stochastic} and Hsu \cite{hsu2002stochastic}.

\subsubsection{Riemannian geometry}
The frame bundle $\cFM$ of a $d$-dimensional smooth manifold is a $d^2+d$ dimensional manifold, where each point $u \in \cFM$ correspond to a point $x \in \M$ together with an ordered basis (frame) of the tangent space $T_x \M$. It is convenient to think of $u$ as a linear bijection $u \colon \R^d \rightarrow T_{\pi(u)}\M$, where $\pi \colon \cFM \rightarrow \M$ denotes the canonical projection. In this sense we obtain an ordered basis $(u(e_1),...,u(e_d))$ of $T_{\pi(u)}\M$ from the canonical basis $(e_1,...,e_d)$ of $\R^d$. A connection in $\cFM$ is a smooth choice of subspaces $\cH_u\cFM \subseteq T_u\cFM$, for each $u\in \cFM$, such that the tangent space at $u \in \cFM$ splits into $T_u\cFM = \cV_u \cFM \oplus \cH_u \cFM$, where $\cV_u \cFM$ denotes the tangent vectors that are tangent to the fibers. The space $\cH_u \cFM$ is called the horizontal tangent space. If we endow $\M$ with a Riemannian structure, that is, a smoothly varying inner product on the tangent spaces, the set of frames can be restricted to orthonormal frames, i.e., the map $u$ being a linear isometry. The resulting subbundle is denoted the orthonormal frame bundle, $\cOM$. We will throughout primarily work with $\cOM$. There is a one-to-one correspondence between the horizontal tangent space at $u$ and the tangent space at $\pi(u)$. This correspondence is described through the restriction of the pushforward map $\pi_*|_{\cH_u\cOM}\colon T_u\cOM \rightarrow T_{\pi(u)}\M$. Furthermore, on the horizontal part of the frame bundle, there exists a set of fundamental horizontal vector fields, $H_1,...,H_d$, defined by $H_i(u) = h_u\left(u(e_i)\right)$, where $h_u = \left(\pi_*|_{\cH_u\cOM}\right)^{-1}$ is the horizontal lift.

It is sometimes convenient to linearize the manifold in the following sense. Let $E \colon \mathbb{R}^d \rightarrow T_p\M$ be an isometric isomorphism from $d$-dimensional Euclidean space to the tangent space at $p \in \M$ and let $\Exp_p \colon T_p\M \rightarrow \M$ denote the exponential map at $p$. Define $\phi \colon D \rightarrow \mathbb{R}^d$ by $\phi := E^{-1} \circ \Exp_p^{-1}$, where $D \subseteq \M$ is the largest subset such that $\Log_p :=\Exp_p^{-1}$ is well-defined, then the pair $(D,\phi)$ is a normal neighborhood centered at $p \in \M$. 

\subsubsection{Horizontal semimartingales - stochastic development}
When the starting frame $U_0$ is fixed, there exists a one-to-one correspondence between Euclidean-, horizontal-, and manifold-valued curves. If $z_t$ is an $\R^d$ curve with derivative $dz_t$, then $H_{dz_t}(u_t) = h_{u_t}(u_t(e_i))dz_t^i =H_i(u_t)dz_t^i$. The construction is called horizontal development and extends to include semimartingales.

 If $Z$ is a continuous Euclidean-valued semimartingale, i.e., a process which decomposes into a sum of a local martingale and an adapted process of locally bounded variation, the horizontal vector fields give rise to an SDE on the frame bundle, driven by the Stratonovich SDE, 
\begin{equation}
	dU_t = H_i(U_t) \circ dZ^i_t, \qquad  U_0=u \in \cFM, \label{eq:horizontal semimartingale sde} 
\end{equation}
where $U$ is a horizontal semimartingale on $\cOM$. The canonical projection, $X:=\pi(U)$, defines a process on $\M$. The process $U$ is called the stochastic development of $Z$ in $\cOM$ and $X$ the stochastic development of $Z$ in $\M$. Similarly, $Z$ is called the anti-development of $X$ and $U$. This way of constructing stochastic processes on manifolds is contributed to Eells, Elworthy, and Malliavin, and is often referred to as the Eells-Elworthy-Malliavin construction. In colloquial terms, the construction is called "rolling without slipping." The intuition behind this terminology originates from rolling a ball along a path drawn with wet ink on a piece of paper. For a detailed description of this construction, one may consult \cite{hsu2002stochastic}. 

Define $\tilde h_u$ as the solution to the stochastic development equation \eqref{eq:horizontal semimartingale sde}, i.e., $\tilde h_u(Z_t) = U_t$. We call this map the stochastic development map. Similarly, we define the stochastic anti-development map as $\tilde h_u^{-1}(U_t) = Z_t$.

\subsubsection{Geometric It\^o formula}
The fundamental theorem of stochastic calculus is It\^ o's formula. The formula generalizes to manifold-valued semimartingales. If $X$ is a continuous manifold-valued semimartingale, $U$ its horizontal lift, and $Z$ its anti-development, then for any smooth function $f \in C^{\infty}(\M)$ the geometric It\^ o formula can be expressed as
	\begin{equation*}\label{ito's formula}
	df(X_t) = U_t(e_i)f(X_t)dZ^i_t + \tfrac{1}{2} U_t(e_j)U_t(e_k)f(X_t)d[Z^j,Z^k]_t.
	\end{equation*}
The formula can equivalently be expressed in terms of the gradient and Hessian as
	\begin{equation*}
	df(X_t) = \left\langle \nabla f(X_t), U_t dZ_t\right\rangle_{X_t} + \tfrac{1}{2} \Hess_{X_t}f(U_t,U_t)d[Z]_t,
	\end{equation*}
 where $\langle \cdot, \cdot \rangle_x$ denotes the Riemannian inner product on $T_x\M$. We simply write $\langle \cdot, \cdot \rangle$ when the subscript is clear from the context.

\subsubsection{Riemannian Brownian bridges}
Brownian bridges on manifolds also have an SDE representation. Whenever the semimartingale $Z$ is a Brownian motion, the resulting process $X$ on $\M$ (resp. $U$ on the horizontal part of $\cOM$) is a Brownian motion. With $p(t,x,v)$ denoting the transition density of the $\M$-valued Brownian motion and $\tilde{p}(t,u,v) = p(t,\pi(u),v)$ its lift to $\cOM$, the corresponding Brownian bridge is the solution of the SDE
	\begin{equation}
		dU_t = H_i(U_t) \circ \left(dB^i_t + \left(U_t^{-1} \left(\pi_*\nabla_{u|_{u=U_t}}^H \log \tilde{p}(t,u,v)\right)\right)^i dt \right), \qquad U_0 = u_0,
        \label{eq:Brown_bridge}
	\end{equation}	 	
where $\nabla^H\tilde{f} = \{H_1 \tilde{f}, \dots, H_d \tilde{f}\}$ is the horizontal gradient. As was the case in the Euclidean setting, the SDE includes the Brownian motion's transition density. However, unlike the Euclidean case for Brownian motion, the transition density is generally not computable in closed-form.

Writing the transition density for the Brownian motion as $p(t,x,v)= p_t(x,v)$, a natural representation of the Brownian bridge is as a time-inhomogeneous Markov process with an infinitesimal generator 
\begin{equation*}
	\tfrac{1}{2}\Delta_{\M} + \nabla \log p_{T-t}(\cdot,v),
\end{equation*}
where $\Delta_{\M}$ is the generalized Laplacian (Laplace-Beltrami) operator on $\M$. We will drop the $\M$ in the sequel whenever referring to the Laplace-Beltrami operator.

\subsubsection{Related constructions}\label{sec:related works}
As the Brownian motion's transition density is only known in a closed-form on few manifolds, other types of bridge processes have been considered (see e.g. \cite{hsu_brownian_1990}). A semi-classical Brownian bridge on a Riemannian, also known as Brownian Riemannian bridge, is a time inhomogeneous strong Markov process with infinitesimal generator 
\begin{equation*}
	\tfrac{1}{2}\Delta + \nabla \log k_{T-t}(\cdot, v),
\end{equation*}
where $k_t$ is the function defined by
\begin{equation*}
	k_t(x,v) := (2 \pi t)^{-n/2} e^{- \frac{r^2(x,v)}{2t}}			J^{-1/2}(x),
\end{equation*}
and $J(x) = |\det D_{\Exp_{v}^{-1}(x)} \Exp_{v}|$ is the Jacobian determinant of the exponential map at $v$ (see any of \cite{elworthy1988geometric,elworthy1982diffusion,elworthy1982stochastic} for more details on semi-classical bridges). This description typically relies on the existence of a pole, i.e. a point in $\M$ where the exponential map maps diffeomorphically to $\M$. The assumption avoids the nuissance of the cut locus. The radial part of the semi-classical bridge has the distribution of a Euclidean valued Brownian bridge. 
 
A generalization beyond the cut locus of the heat kernel formula described by Elworthy and Truman \cite{elworthy1982diffusion} is due to Thompson \cite{thompson2015submanifold}. Let $N$ be a closed embedded submanifold on $\M$ and define the distance function to $N$ by $r_N(\cdot) := d(\cdot, N)$, then introduce the diffusion on $\M$ with time-dependent infinitesimal generator
\begin{equation*}\label{fermi bridge infinitesimal generator}
	\frac{1}{2}\Delta_{\M} - \frac{\nabla r^2_N}{2(T-t)}\quad =\quad \frac{1}{2}\Delta_{\M} - \frac{r_N}{T-t} \frac{\partial}{\partial 			r_N},
\end{equation*}
where $\tfrac{\partial}{\partial r_N}$ denotes differentiation in the radial direction and $\Delta$ denotes the Laplace-Beltrami operator. This diffusion process is called a Fermi bridge (see \cite{thompson2015submanifold,thompson2018brownian}). More generally, \cite{li2016generalised} defined generalised Brownian bridge processes, between $x_0$ and $x_T=v$ with terminal time $T$, as a Markov process $(x_t)_{t \geq 0}$ with infinitesimal generator
\begin{equation*}
	\tfrac{1}{2}\Delta_{\M} - f(t) \nabla \tfrac{r^2_{v}}{2},
\end{equation*}
where $f$ is a suitably smooth real valued function defined on $[0,1)$ satisfying $\lim_{t \uparrow T} f(t) = \infty$ and $\lim_{t \uparrow T} x_t = x_T$ almost surely. We also refer the reader to \cite{li2016hypo} for a description of hypoelliptic bridges on manifolds.

\section{A Girsanov change of measure}\label{sec:girsanov change of measure}
Intractable transition densities complicate exact simulation of the conditioned processes: For example, it is not possible to directly simulate from \eqref{eq:Brown_bridge} if $\tilde p$ cannot be computed to get the score and the corresponding drift. Various methods exist to approximate these processes \cite{bladt2014simple, bladt2016simulation, clark1990simulation, delyon_simulation_2006, schauer2017guided, van2017bayesian}. Common for all these method is that they rely on a change of measure argument and that the changed measure respects the original measure, in the sense of absolute continuity. 

In this section, we recall a Cameron-Martin-Girsanov result for manifold valued processes, see \cite{elworthy1988geometric}, that we will need later on. We assume throughout that $X_t = \pi(U_t)$ is non-explosive. This is for example the case if $\M$ is compact. Let $dM_t = \sigma_t(Z_t)dW_t$ be the martingale part of $Z$, where we have assumed that $\sigma$ is invertible. Define the local martingale $L_t$ by the equation 
    \begin{equation*}
L_t := -\int_0^t \left\langle 						\frac{\nabla^H_{u|_{u=U_s}} \tilde r^2_v(u)}{2(T-s)},  U(e_i) \right\rangle dM_s^i.
    \end{equation*}
Using that $U$ is an isometry, the local martingale is identical to the process
	\begin{equation*}
L_t := -\int_0^t \left\langle 						\sigma^{-1}_s(\tilde h_u^{-1}(U_s))U_s^{-1}\left(\frac{\nabla^H_{u|_{u=U_s}}\tilde r^2_v(u)}{2(T-s)}\right),  e_j \right\rangle dW_s^j.
		\end{equation*}
The corresponding Radon-Nikodym derivative is given by
    \begin{equation}\label{eq:radon-nikodym derivative for radial semimartingale}
	D_t = \exp\bigg[L_t
		 - \frac{1}{2}
		 \int_0^t \norm[\bigg]{									\sigma^{-1}_s(\tilde h_u^{-1}(U_s))U^{-1}_s\left(\frac{\nabla^H_{u|_{u=U_s}}\tilde r^2_v(u)}{2(T-s)}\right)}^2ds \bigg].
	\end{equation}
The Novikov condition ensures that the measure $\Q_t$, defined as $d\Q_t = D_t d\Pb$, is equivalent to the measure $\Pb$ on the time interval $[0,T)$. As a consequence of Girsanov's theorem, we have that the $\Pb_t$-Brownian motion, $W_t$, satisfies the equation
	\begin{equation*}
dW_t = d\hat{W}_t - \sigma^{-1}_t(\tilde h_u^{-1}(U_t))U_t^{-1}\left(\frac{\nabla^H_{u|_{u=U_s}} \tilde r^2_v(u)}{2(T-t)}\right)dt,
	\end{equation*}
and we have that $\hat W_t$ is a $\Q_t$-Brownian motion. Plugging the expression for $W_t$ into equation \eqref{eq:horizontal semimartingale sde} yields the SDE 
			\begin{equation}
						dU_t 
						=
						 H_i(U_t) \circ \biggl(dZ^i_t - 						\left(U_t^{-1}\left(\frac{\nabla^{H}\tilde{r}^2_v(U_t)}						{2(T-t)}\right)\right)^i dt\biggr) .
						 \label{eq:radial bridge horizontal sde}
			\end{equation}
We will consider solutions to this SDE and its projection to $\M$ throughout the paper.
We term the solution of \eqref{eq:radial bridge horizontal sde} the \textit{radial bridge process}.
The added drift can equivalently we written $-\frac{H_i \tilde r^2_v(U_t)}{2(T-t)}$
because
$
        \left(U_t^{-1}\left(\pi_*\nabla^H \tilde r_v^2(U_t)\right)\right)^i = \left\langle \nabla r_v^2(X_t) , U_t(e_i)\right\rangle = H_i \tilde r_v^2(U_t)$.

	\begin{example}
	Consider the case where the driving semimartingale is a standard Brownian motion. The Radon-Nikodym derivative \eqref{eq:radon-nikodym derivative for radial semimartingale} simplifies to
		\begin{equation}\label{eq: radon-nykodym derivative for BM}
    	D_t := \exp\bigg[-\int_0^t \frac{r_N(X_s)}{T-s}\bigg\langle 			\frac{\partial}{\partial r_N}, U_sdB_s\bigg\rangle - \frac{1}{2}		\int_0^t \frac{r_N^2(X_s)}{(T-s)^2}ds\bigg],
        \end{equation}
almost surely, since the norm of the gradient of the radial process is constantly one away from the cut locus and its starting point, $\Cut(x_0)\cup \{x_0\}$, and since this set is polar for $X$. In particular, as mentioned in \cite{thompson2018brownian}, the integral 
	\begin{equation*}
		\int_0^t\bigg\langle \frac{\partial}{\partial r_N}, U_sdB_s\bigg\rangle = \beta_t
	\end{equation*}
is a one-dimensional standard Brownian motion, which follows from Levy's characterization theorem of Brownian motions and that $U$ consists of isometries.
	\end{example}

\section{Bridge simulation on manifolds}\label{sec:2 main results - bride simulation on manifolds}
This section initiates horizontal semimartingales from Euclidean-valued semimartingales and describes the guided semimartingales considered in this paper. Throughout, we assume that $Z = (Z^1, Z^2,...,Z^d)$ is a Euclidean-valued semimartingale given by
	\begin{equation}\label{eq:z semimartingale}
	dZ_t^k = a^k(t,Z_t)dt + \sigma_{m}^k(t,Z_t)dW_t^m,
	\end{equation}	
where $W$ is a Euclidean-Brownian motion and $a$ and $\sigma$ are suitably integrable maps. Horizontal semimartingales are then obtained as solutions to \eqref{eq:horizontal semimartingale sde} which evidentally happens if and only if, for all smooth functions $\tilde{f}$ on $\cOM$, it holds that
	\begin{equation*}\label{eq:solution to sde on frame bundle for all smooth functions}
	d\tilde{f}(U_t) = H_i \tilde{f}(U_t) \circ dZ^i_t.
	\end{equation*}
By the It\^o-Stratonovich conversion, we obtain the corresponding It\^o equation
	\begin{equation*}
	d \tilde f(U_t) 
	= 
	 \cOML_u \tilde{f}(U_t) dt + \sigma_m^k (t,\tilde h_u^{-1}(U_t))H_k\tilde{f}(U_t)dW^m_t,
	\end{equation*}
where $\cOML_u$ is an operator in $\cOM$ defined by
	\begin{equation}\label{eq:horizontal diffusion operator}
	\cOML_u \tilde f(U_t) := a^k(t,\tilde h_u^{-1}(U_t))H_k \tilde f(U_t) + \tfrac{1}{2} C^{ij}(t,\tilde h_u^{-1}(U_t))H_iH_j\tilde f(U_t),
	\end{equation}
with initial frame $u \in \cOM$, and where $C^{ij} := \left(\sigma \sigma^T\right)_{ij}$ denotes the $ij$'th entrance of $\sigma\sigma^T$.
By definition of Bochner's horizontal Laplacian, $\Delta_{\cOM}$, whenever $\sigma$ is orthogonal \eqref{eq:horizontal diffusion operator} simplifies to $\cOML =  \tilde{V} + \tfrac{1}{2}\Delta_{\cOM}$, where $\tilde V$ is the horizontal lift of a vector field on $\M$.
Adding a drift term pointing in the radial direction gives rise to a new operator on $\cOM$
	\begin{equation}\label{eq:guided diffusion operator}
	\cOML_{u} - \frac{\nabla^H \tilde{r}_v^2}{2(T-t)}.
	\end{equation}
We further define $\cL_u$ as the operator in $\M$ that satisfies $\cOML_u \tilde f(u) = \cL_u f(x)$, for all $f \in C^{\infty}(\M)$, where $\tilde f = f \circ \pi$. As we shall see below, this added drift term acts as a guiding term that forces the process towards its target. 
The operator \eqref{eq:guided diffusion operator} on $\cOM$ can equivalently be expressed as the generator of the SDE \eqref{eq:radial bridge horizontal sde}.

\begin{example}
Whenever $C$ is the identity matrix and $a$ vanishes, i.e., when $Z$ is a standard Brownian motion, $\cOML_u$ becomes the horizontal Laplacian, $\Delta_{\cOM}$, identified by
	\begin{equation*}
	\Delta_{\cOM} \tilde{f}(u) = \Delta_{\M} f(x),
	\end{equation*}
for all $f \in C^{\infty}(\M)$, where $\tilde{f} = f \circ \pi$ is the lift of $f$ to $C^{\infty}(\cOM)$. 
\end{example}

\subsection{Guided bridges}\label{sec: the main result}
We can now establish a method to simulate conditioned semimartingales on a Riemannian manifold $\M$, starting from some fixed initial point $x_0 \in \M$ and conditioned to be at $v \in \M$ at time $T$, for $T > 0$. Utilizing stochastic development, we obtain manifold valued semimartingales from Euclidean valued ones by solving SDEs on the orthonormal frame bundle, $\cOM$. Throughout, we assume a strong solution to \eqref{eq:horizontal semimartingale sde}.

We recall from Girsanov's theorem that any $\Pb$-local martingale $M_t$ becomes a $\Qb$-local martingale $\tilde{M}_t$ under a suitable change of measure (see e.g. \cite{revuzyor}). Consequently, approximating the measure of a conditioned process relates naturally to Girsanov's theorem.

Define $A := C^{-1} = (\sigma \sigma^T)^{-1}$ and the function $g \colon \R_+ \times \R_+ \times \R^d \times \R^d \rightarrow \R$ by
    \begin{equation}\label{eq:g function}
        g(t,r,z,\xi) := \frac{r^2}{T-t}\lVert \sigma^{-1}_t(z)\xi\rVert^2.
    \end{equation}
Let $(D,\phi)$ be a normal neighborhood centered at $v \in \M$. On the set $\phi(D)$ the function $g$ is smooth. The following lemma provides an It\^o expansion of $g$.

\begin{lemma}\label{lemma: g function}
Let $\xi = U^{-1}(\nabla r_v)$. On the set $\phi(D)$, $g$ as defined in \eqref{eq:g function} has the It\^ o SDE expression
\begin{align*}
& dg(t,r_t,Z_t,\xi_t) = E(t) dt  + \frac{r_t^2}{T-t}\xi^T_t d\left(A(t,Z_t)\right)\xi_t + F(t) dr_t + J_i(t) d\xi^i_t \nonumber \\& \quad +  \frac{1}{2}\left(G(t) d[r,r]_t +  J_{ij}(t)d[\xi^i,\xi^j]_t  \right) + H_i(t)d[r,Z^i]_t + I_j(t) d[r,\xi^j]_t+K_{ij}(t)d[\xi^i,Z^j]_t
\end{align*}
where 
\begin{align*}
    &A(t,z):=\big(\sigma_t(z)\sigma_t(z)^T)^{-1};\ 
	E(t) 
	:= \partiel t g(t,r,z,\xi)=
	 \frac{r^2}{(T-t)^2}\norm{\sigma^{-1}_t(z)\xi}^2; 	
     \\
    &F(t) 
        := \partiel r g(t,r,z,\xi) 
	= 
	2\frac{r}{T-t}\norm[\big]{ \sigma^{-1}_t(z)\xi}^2;\ 
	G(t) 
	:= \frac{\partial^2}{(\partial r)^2} g(t,r,z,\xi) =   
	2\frac{1}{T-t}\norm[\big]{ \sigma^{-1}_t(z)\xi}^2; 
    \\ 
           &H_j(t) 
    :=
	\frac{\partial^2}{\partial r \partial z_j}
	g(t,r,z,\xi)
	= 	
	2\frac{r}{T-t}\xi^T\frac{\partial}{\partial z_j}\left(A(t,z)\right)\xi;
	\\	
           &I_j(t) 
	 := 
	 \frac{\partial^2}{\partial r \partial \xi_j}g(t,r,z,\xi)=
	4\frac{r}{T-t}\xi^TA(t,z)e_j; 
	\ J_j(t) 
	:= \frac{\partial}{\partial \xi_j}g(t,r,z,\xi) 
	 2\frac{r^2}{T-t}\xi^TA(t,z)e_i;
	 \\
           &J_{ij}(t) 
	  :=\frac{\partial^2}{\partial \xi_i \partial \xi_j}g(t,r,z,\xi) =  
	  2\frac{r^2}{T-t}e_jA(t,z)e_j; 
      \\ 
           &K_{ij}(t) 
      := \frac{\partial^2}{\partial \xi_i \partial z_j}g(t,r,z,\xi) =  
	 2\frac{r^2}{T-t}\xi^T\frac{\partial}{\partial z_j}\left(A(t,z)\right)e_i; 
\end{align*}
\end{lemma}
\begin{proof}
Follows from application of the multidimensional It\^ o's formula.
\end{proof}
\begin{remark}
In the case where the diffusion parameter $\sigma$ is the identity matrix, \eqref{eq:g function} simplifies to $g(t,r) = r^2/(T-t)$. In this case, the It\^o expression for $g$, by stochastic integration by parts, is simply 
\begin{equation*}
dg(t,r_t) = \frac{r^2_t}{(T-t)^2}dt + 2\frac{r_t}{(T-t)}dr_t + \frac{1}{(T-t)}dt = \frac{r^2_t}{(T-t)^2}dt + \frac{1}{(T-t)}dr_t^2.
\end{equation*}
\end{remark}	

We can now state the main result of the paper which is a generalization of Delyon \& Hu \cite{delyon_simulation_2006} to Riemannian manifolds.
\begin{theorem}\label{main result theorem}
		Suppose that $Z$ is a semimartingale defined by \eqref{eq:z semimartingale} and $U$ the horizontal semimartingale defined by \eqref{eq:horizontal semimartingale sde} with $X=\pi(U)$. Let $V$ be the solution to \eqref{eq:radial bridge horizontal sde} and let $Y = \pi(V)$ denote the canonical projection onto $\M$. The law of the process $(Y_t)_{0\leq t<T}$ is equivalent to the law of the conditioned process $(X_t)_{0 \leq t<T}$, conditioned at $X_T=v$.  We have that 
			\begin{equation}\label{conditional mean}
				\E[f(X)|X_T = v] =  \lim_{t \uparrow T} \frac{\E[f(Y) \varphi_t]}{\E[ \varphi_t]},
			\end{equation}
where the likelihood $\varphi_t$ has the form
		\begin{align}\label{eq: generalized phi term}
            & -2d\log \varphi_t = \frac{r_t^2}{T-t}\xi^T_t d\left(A(t,Z_t)\right)\xi_t + F(t) dr_t + J_i(t) d\xi^i_t \\& \quad +  \frac{1}{2}\left(G(t) d[r,r]_t +  J_{ij}(t)d[\xi^i,\xi^j]_t  \right) + H_i(t)d[r,Z^i]_t + I_j(t) d[r,\xi^j]_t+K_{ij}(t)d[\xi^i,Z^j]_t.
\nonumber 
		\end{align}
\end{theorem}
\begin{proof}
Here, we present the main structure of the proof. The constituent parts will be made rigorous in the sections that follow. By a proper change of measure, through Girsanov's theorem, one obtains the SDE \eqref{eq:radial bridge horizontal sde} from \eqref{eq:horizontal semimartingale sde}. The change of measure is valid on the interval $[0,t]$, for every $t \in [0,T)$. A decomposition of the Radon-Nikodym derivative leads to a term only involving the radial process and the time to the termination. By invoking an $L^2$ bound on the radial bridge process, we show the radial bridge process's almost sure convergence to the desired point (\ref{sec:almost sure convergence}). The result is concluded by an argument (\refsec{sec:proof of main result}) similar to Delyon \& Hu \cite[~Lemma 7]{delyon_simulation_2006}
\end{proof}

\subsection{The case of Brownian motion}\label{sec: Brownian motion case}
A special case of the above result is when the $X$ is a Brownian motion. We state the result separately below as it is simpler and because, in this case, we can remove the limit on the right-hand side of \eqref{conditional mean}. Let $\Theta_x(y) = |\det D_y \Exp_x|$ be the Jacobian determinant at $y$ of the Jacobi field along the geodesic from $x$ to $y$.

\begin{theorem}\label{main result theorem for BM}
	Let $U$ be the solution of \eqref{eq:horizontal semimartingale sde}, with $Z$ a standard Brownian motion in $\R^d$ and let $X$ be the canonical projection onto $\M$. Furthermore, let $V$ be the solution of \eqref{eq:radial bridge horizontal sde}.		
 the conditioned law of $X$ given $X_T = v$ is absolutely continuous with respect to the law of $Y$ on $[0,t]$, for all $t < T$.
	 With the notation as in \refthm{main result theorem}, we have 
		\begin{equation*}\label{eq: conditional mean for BM}
			\E[f(X)|X_T = v] = C \E\left[f(Y)\varphi_T\right],
		\end{equation*}
	where $C>0$ is a constant. In particular, the likelihood ratio $\varphi_t$ in \eqref{eq: generalized phi term} simplifies to
	\begin{equation*}
		d\log \varphi_t = 
\frac{r(Y_t)}{T-t} \left(d\eta + dL_s\right),
	\end{equation*}
	where 
\begin{equation*}
d\eta_s = \frac{\partial}{\partial r} \log \Theta_v^{-\frac{1}{2}}ds
\end{equation*}
is supported on $\M\backslash \Cut(v)$ and $L_s$ is the geometric local time at $\Cut(v)$.
			
\end{theorem}
\begin{example}
    In the situation where $\M = \R^d$ and $X$ is a Brownian motion, the result matches Theorem 5 of Delyon \& Hu \cite{delyon_simulation_2006}. 
    No curvature or cut locus exists in the Euclidean setting and, therefore, the likelihood ratio $\varphi_t=1$. 
\end{example}
	
Below, we prove the remaining parts of the two theorems. In particular, we shall derive the expression of $\varphi_t$ in \eqref{eq: generalized phi term} as well as review the behavior of the radial process of a continuous semimartingale on $\M$.

\section{Radial part of manifold valued semimartingales}\label{sec:radial part of semimartingale}
To analyze the guided process behavior, we need to control the radial part of the guided process. 
Barden and Le described the behavior of the radial process for semimartingales on manifolds \cite{barden1997some,le1995ito}, generalizing a result by Kendall \cite{kendall1987radial}, which describes the radial behavior of a Brownian motion. To take full advantage of Barden and Le's result, we describe the geometry of the cut locus and present their result.

Let $x_0$ be a given reference point on a complete Riemannian manifold $\M$. We let $\mathcal{Q}(x_0)$ denote the set of points in $T_{x_0}\M$ where the exponential map is singular, that is, points $v \in T_{x_0}\M$ such that $D(\exp_{x_0})(v)$ has rank $k < d$ (for a $d$ dimensional manifold $\M$). Define $Q(x_0):=\exp_{x_0}(\mathcal{Q}(x_0))$ as the image under $\exp_{x_0}$. We say that $\mathcal{Q}(x_0)$ and $Q(x_0)$ are the conjugate loci of $T_{x_0}\M$ and $\M$, respectively. Let $\Cut(x_0)$ denote the cut locus of $x_0$ in $\M$. Then $\Cut(x_0) \cap Q(x_0)$ denote the conjugate part of the cut locus of $x_0$.

A $(d-1)$-dimensional submanifold $S \subset \M$ is called two-sided if its normal bundle is trivial. By Barden and Le \cite[Theorem 2]{le1995ito}, the cut locus, except for a subset of Hausdorff $(d-1)$-measure zero,  is described as a countable disjoint union, $\mathscr L$, of open two-sided $(d-1)$-dimensional submanifolds.
Define $E$ to be the union of points consisting of $\Cut(x_0) \cap Q(x_0)$ and the points in $\Cut(x_0) \backslash Q(x_0)$ where at least $3$ minimal geodesics to $x_0$ exists. The set $E$ has Hausdorff $(d-1)$-measure zero. A set $A$ is said to be a polar set for $X$ if the first hitting time of $A$ by $X$ is almost surely infinite. 

Suppose that the process $U$ is the solution of \eqref{eq:horizontal semimartingale sde}, that $E$ is a polar set of $X$, and that $\Pb(X_t = x, 0 < t \leq \infty) = 0$. Furthermore, let $\mathring C$ be a disjoint union of $\Cut(v)$, which consist of countably many smooth connected two-sided $(d-1)$-dimensional submanifolds, where $\mathring C_i$ denote the connected components of $\mathring C$, and let $D \subseteq \M$ be a regular domain. We denote by $D_\pm f_{x}(\pm \nu)$ the one-sided G\^ateaux derivatives of $f$, defined by
	\begin{equation*}
	D_+ f_{x}(\nu):= \lim_{\varepsilon \downarrow 0} \tfrac{1}{\varepsilon}\left(f\left(\exp_x(\varepsilon \nu)\right) - f(x)\right), \qquad D_- f_x(\nu) := - D_+ f_x(-\nu),
	\end{equation*}
for all $x \in \M$ and $\nu \in T_x\M$.
From \cite[Theorem 3]{le1995ito},  we get a formula for the radial process of $X= \pi(U)$ as (see also \cite{thompson2015submanifold} for a more general formula)
	\begin{align*}
	r_v\left(X_{t\wedge \tau_D}\right) =  r_v(x) 
	&+
	 \int_0^{t\wedge \tau_D} 1_{\{X_s \notin \mathring C\}} \left\langle \nabla r_v(X_s),U_sdZ_s\right\rangle
	 \\
	&+
	\int_0^{t\wedge \tau_D}1_{\{X_s \in \mathring C_i\}} \left\langle \nabla \left(r_v\circ P^i\right)(X_s),U_sdZ_s\right\rangle
	\\
	&+
	 \frac{1}{2}\int_0^{t\wedge \tau_D}1_{\{X_s \notin \mathring C\}} \Hess_{X_s} r_v(U_s,U_s)d[Z,Z]_t 
	 \\
	 &+ 
	 \frac{1}{2}\int_0^{t\wedge \tau_D}1_{\{X_s \in \mathring C_i\}} \Hess_{X_s} \left(r_v\circ P^i \right)(U_s,U_s)d[Z,Z]_t 
	 \\
	  &+ \frac{1}{2}\int_0^{t\wedge \tau_D} \left(D_+r_{X_s}(\nu)dL_s^{+\nu,\mathring C}(X) + D_+r_{X_s}(-\nu)dL_s^{-\nu,\mathring C}(X)\right) 
	  \\
	  &+ L^0_{t \wedge \tau_D}\left(r_v(X)\right),
	\end{align*}
for all $t \geq 0$, almost surely, where $\tau_D$ is the first exit time of $D$ by $X$, $P^i$ the orthogonal projection onto the connected component $\mathring C_i$ of $\mathring C$, $L^0$ the local time at $0$ of $r_v(X)$, $dL^{\pm \nu, \mathring C}$ the associated random measures of the  geometric local times at $\mathring C$ as defined in \cite{thompson2015submanifold}, and $D_\pm r_{X_t}(\pm \nu)$ are the one-sided G\^ateaux derivatives of $r$ along $\pm\nu$, the unit normal vector-field of $\Cut(x_0)\backslash E$.

Under the assumptions made in \refsec{sec:notation and convention}, the terms which depend on the projections vanish in the equation above. In this case, the formula for the squared radial process becomes
	\begin{align}\label{eq:squared radial formula}
	r_v^2\left(X_{t\wedge \tau_D}\right) =  r_v^2(x) \
	&+
	 \int_0^{t\wedge \tau_D}  \left\langle \nabla r_v^2(X_s),U_sdZ_s\right\rangle 
	+
	 \frac{1}{2}\int_0^{t\wedge \tau_D} \Hess_{X_s} r_v^2(U_s,U_s)d[Z,Z]_t 
     \nonumber
	 \\
	  &+ \frac{1}{2}\int_0^{t\wedge \tau_D} r_v(X_s)\left(D_+r_{X_s}(\nu) + D_+r_{X_s}(-\nu)\right)d\tilde L_s^{\mathring C}(X),
\end{align}
where $d\tilde L^{\mathring C}$ is the random measure carried by $\mathring C$,  associated to the continuous predictable non-decreasing functional process. Furthermore, as defined in \cite{barden1997some}, we have $D_+r_{X_s}(\nu) + D_+r_{X_s}(-\nu) \leq 0$ and so we define the random measure
\begin{equation*}
dL^{\mathring C}_t = -\left(D_+r_{X_s}(\nu) + D_+r_{X_s}(-\nu)\right)d\tilde L_s^{\mathring C}(X).
\end{equation*}

	\begin{example}
		The radial part of a Brownian motion, which is due to Kendall \cite{kendall1987radial}, has the following representation:
		\begin{equation}\label{radial representation by kendall}
	dr_t = \tfrac{1}{2}\Delta_M r_t dt + d\beta_t - dL_t,
\end{equation}
where $L_t$ is the local time at the cut locus and $\beta$ is a one-dimensional real valued standard Brownian motion.
	\end{example}

\subsection{Properties of the radial bridge}\label{sec:properties of the radial bridge}

The perhaps most fundamental property of any bridge proposal process is that it converges to the correct point almost surely. The radial bridge process with infinitesimal generator given by \eqref{eq:guided diffusion operator} has a drift term that always points in the radial direction. This drift acts as a pulling term in the radial direction, and it ensures that the radial bridge process converges to the desired endpoint. 
 
\subsubsection{Bound for the radial bridge}\label{sec:an L^2 bound for the radial bridge}
Throughout this section, let $U$ be the process generated by \eqref{eq:guided diffusion operator} and let $X = \pi(U)$ be the projection of $U$ on $\M$. Then by Barden and Le's formula \eqref{eq:squared radial formula} together with the geometric It\^ o formula for semimartingales, we have
	\begin{align*}
        &r_v^2(Y_{t \wedge \tau_D}) = \\&
        \quad r_v^2(x_0) + 2\int_0^{t \wedge \tau_D} r_v(Y_s)dN_s 
		- \int_0^{t \wedge \tau_D} r_v(Y_s) dL^{\mathring{C}}_s  
		  + \int_0^{t \wedge \tau_D} \left(\cL_z r^2_v(Y_s) -  2\frac{r^2_v(Y_s)}{T-s}\right)ds,
		\end{align*}
where $N_t$ is a one-dimensional Euclidean martingale defined by $N_t:= \int_0^t \sigma^k_m(s,Z_s) \left\langle \nabla r, U_s(e_k)\right\rangle dW^m_s$, and $\cL_z$ is defined in \eqref{eq:horizontal diffusion operator}. As the second term above is a local martingale, we have
	\begin{align}
    &\E\left[r_v^2\left(Y_{t \wedge \tau_{D_i}}\right)\right]  =  \\&
        \quad r_v^2(x_0) -  \E\left[\int_0^{t \wedge \tau_D} r_v(Y_s) dL^{\mathring{C}}_s\right]
	+ \E\left[\int_0^{t \wedge \tau_D} \cL_z r^2_v(Y_s) ds \right] \nonumber 
		-  2\E\left[\int_0^{t \wedge \tau_D}\frac{r^2_v(Y_s)}{T-s}ds\right]
	\end{align}
for all $t \in [0,T)$ almost surely. In particular, under the assumption of \eqref{assumption of bounded elliptic operator} below, the two last terms can be rewritten by Lebesgue's dominated convergence and Fubini's theorem such that
	\begin{equation*}\label{equation: mean of squared radial process}
        \begin{split}
    &\E\left[r_v^2\left(Y_{t \wedge \tau_{D_i}}\right)\right] 
        = \\&\quad
	r_v^2(x_0) -  \E\left[\int_0^{t \wedge \tau_D} r_v(Y_s) dL^{\mathring{C}}_s\right]
	+ \int_0^{t}\E\left[1_{(s < \tau_D)} \cL_z r^2_v(Y_s) \right]ds  
		-  2\int_0^{t}\E\left[1_{(s < \tau_D)}\frac{r^2_v(Y_s)}{T-s}\right]ds
        \end{split}
	\end{equation*}
for all $t \in [0,T)$, almost surely.

\begin{theorem}\label{L2 bound of radial process}(Adapted from theorem 3.1 in \cite{thompson2018brownian})
		Let $D$ be a regular domain (smooth boundary and compact closure) in $\M$, and $\tau_D$ be the first exit time of $Y$ from $D$. Suppose there exist constants $\nu \geq 1$ and $\lambda \geq 0$ such that
\begin{equation}\label{assumption of bounded elliptic operator}
	\cL_z r_v^2 \leq \nu + \lambda r^2_v
\end{equation}
on $ D \backslash \Cut(p)$. Then we have
\begin{equation}\label{mean squared radial bound}
	\E [1_{\{t < \tau_D\}} r_v^2(Y_t)] 
    \leq
    \left( r^2_v(x_0) + \nu t\left(\frac{t}{T-t}\right) \right)\left(\frac{T-t}{t}\right)^2 e^{\lambda t},
\end{equation}
for all $t \in [0,T)$.
\end{theorem}
	
\begin{proof}
Define
    $
	f(t) := \E\left[1_{\left(t < \tau_{D_i}\right)}r_v^2(Y_{t}) \right]
    $
and note that 
	\begin{equation*}
	 f(t) := \E\left[1_{\left(t < \tau_{D_i}\right)}r_v^2(Y_{t}) \right]= \E\left[r_v^2\left(Y_{t \wedge \tau_{D_i}}\right)\right] - \E\left[1_{\left(t \geq \tau_{D_i}\right)} r_v^2(Y_{\tau_{D_i}}) \right].
	\end{equation*}
Since the maps
	\begin{align*}
	t & \mapsto \E\left[\int_0^{t \wedge \tau_D} r_v(Y_s) dL^{\overset{\circ}{C}}_s\right], \quad
	t  \mapsto \E\left[1_{\left(t \geq \tau_{D_i}\right)} r_v^2(Y_{\tau_{D_i}}) \right],
	\end{align*}
are non-decreasing, their derivatives are non-negative. Differentiating the function $f(t)$, it follows that
	\begin{align*}
	\frac{d}{dt}f(t) \leq & \frac{d}{dt} \E\left[r_v^2\left(Y_{t \wedge \tau_{D_i}}\right)\right] \\
	 \leq & \E\left[1_{(t < \tau_{D_i})} \cL_z  r^2_v(Y_t) \right] 
		-  2\E\left[1_{(s < \tau_{D_i})}\frac{r^2_v(Y_s)}{T-s}ds\right]\\
	 \leq & \nu + \left(\lambda- \frac{2}{T-t}\right) f(t),
	\end{align*}
where the second inequality follows by \eqref{equation: mean of squared radial process} and the third inequality follows from \eqref{assumption of bounded elliptic operator} and Leibniz' rule. An application of Gronwall's inequality yields the claim
\begin{align*}
	f(t) \leq & \left(r_v^2(x_0) + \nu\int_0^t \left(\frac{T}{T-s}\right)^2e^{-\lambda s} ds \right)\left(\frac{T-t}{t}\right)^2 e^{\lambda t} \\
	\leq & \left( r^2_v(x_0) + \nu t\left(\frac{t}{T-t}\right) \right)\left(\frac{T-t}{t}\right)^2 e^{\lambda t}.
\end{align*}
\end{proof}

\subsubsection{Almost sure convergence}\label{sec:almost sure convergence}
We now show that the guided bridge actually hits its target.
\begin{proposition}\label{prop:almost sure convergence}
The proposal bridge process $Y$ satisfy the bridge property 
\begin{equation*}
	\lim_{t \uparrow T} r_v(Y_t)= 0,
\end{equation*}
$\Qb$-almost surely, where $\Qb$ is the extension of $\Qb_t$.
\end{proposition}
\begin{proof}
    We only need to show that $\Q(r_v(Y_T)=0)=1$.
Let $\{D_n\}_{n=1}^{\infty}$ be an exhaustion of $\M$, that is, a sequence consists of open, relatively compact subsets of $M$ such that $\bar{D}_n \subseteq D_{n+1}$ and $\M = \bigcup_{n=1}^{\infty} D_n$. Furthermore, let $\tau_{D_n}$ denote the first exit time of $X$ from $D_n$. Then, from \refthm{L2 bound of radial process}, the sequence 
$
	\bigl(\E [1_{\{t < \tau_{D_n}\}} r_v^2(Y_t)]\bigr)_{n=1}				^{\infty}
$
is non-decreasing and bounded. Hence, from the monotone convergence theorem, it has a limit which is bounded by the right-hand side of \eqref{mean squared radial bound}. Applying Jensen's inequality to the left-hand side of \eqref{mean squared radial bound},
\begin{equation*}
	\E [ r_v(Y_t)] \leq \left( r^2_v(x_0) + \nu t\left(\frac{t}{T-t}\right) \right)^{\frac{1}{2}}\left(\frac{T-t}{t}\right) e^{\frac{\lambda t}{2}}.
\end{equation*}
By the monotone convergence theorem, the right-hand side of the equation above converges to zero, and since $r_v(Y_T)$ is non-negative, we conclude that $\E[r_v(Y_t)] = 0$ and hence $\Q(r_v(Y_T)=0)=1$.
\end{proof}

\section{Proof of Theorem \ref{main result theorem} and \ref{main result theorem for BM}}\label{sec:proof of main result}
We now prove the remaining parts of the main theorems, starting with technical lemmas. The first result provides an SDE expression for the pullback of the radial vector field.

\subsection{Pullback of radial field}\label{sec:pullback process of radial field}
We follow the theory and notation in \cite[Section 2.4]{sommer2016anisotropically} and \cite[Section 11.5]{taubes2011differential}. If $s$ denotes a local vector field on $\M$, we can define a map $s^{\cFM} \colon \cFM \rightarrow \R^d$ by $s^{\cFM}(u) = u^{-1}s(\pi(u))$. Now, if $x_t$ is a curve on $\M$ and $w_t$ is the horizontal lift of $x_t$ from $w$,
we can let $s_t=w_{t,i}s^i_t$. Then $s^{FM}(w_t)=(s^1_t,\ldots,s^d_t)^T$ and
\begin{equation}
  (s^{FM})_*(h_{w_t}(\dot x_t))
  =
  w_t^{-1}\nabla_{\dot x_t}s
  =\frac{d}{dt}
  (s^1_t,\ldots,s^d_t)^T
  \ ,
  \label{eq:framecovder}
\end{equation}
where $h_w$ is the horizontal lift operator. In other words, the covariant derivative takes the form of the standard derivative applied
to the frame coordinates $s^i_t$.

	\begin{lemma}\label{lemma:sde for Y}
	Let $s$ be a local vector field on $\M$, $K_t = U^{-1}_t\left( s\left(\pi\left(U_t\right)\right)\right)$, and $dU_t = H_i(U_t) \circ dZ^i_t$. 
	Then on $\M \backslash \Cut(p)$
		\begin{equation*}
		dK_t = U_t^{-1}\left(\nabla_{U_te_i}s\left(\pi\left(U_t\right)\right)
 \right) \circ dZ_t^i.
		\end{equation*}
	The corresponding It\^o equation is given by 
	\begin{equation*}
	dK_t =
	 U_t^{-1}\left(\nabla_{U_te_i}s\left(\pi\left(U_t\right)\right)
 \right)  dZ_t^i + \tfrac{1}{2} U_t^{-1}\left(\nabla_{U_t e_j}\nabla_{U_te_i}s\left(
\pi\left(U_t\right)\right)\right)d[Z^j,Z^i]_t.
	\end{equation*}
	\end{lemma}
	
	\begin{proof}
We can use the first equality in \eqref{eq:framecovder}, with $\dot{x}_0=ue$, for some $u\in \pi^{-1}(x_0)$, and $e\in\mathbb R^d$, and $w_0=u$, to get
\begin{equation*}
  (s^{FM})_*(h_u(ue))
  =
  u^{-1}\nabla_{ue}s
  =
  u^{-1}\nabla_{ue}\frac{\nabla d(v,\cdot)^2}2
  \ .
  \label{}
\end{equation*}
Now, $dU_t=H_i(U_t)\circ dZ_t^i$ and therefore 
\begin{equation*}
  d\big(u^{-1}s(\pi(U_t))\big)
  =
  d\big(s^{FM}(U_t)\big)
  =
  (s^{FM})_*\big(H_i(U_t)\big)\circ dZ_t^i
  =
  U_t^{-1}\nabla_{U_te_i}s
  \circ dZ_t^i
  \ ,
  \label{}
\end{equation*}
since $s^{FM}$ is a local diffeomorphism.

The second claim follows by an application of the first claim on the local vector field $\nabla_{U_te_i}s$ to get an sde for $U^{-1}_t\left(\nabla_{U_te_i}s\left(\pi(U_t)\right)\right)$. 
	\end{proof}

 The following result is an adaptation of Lemma 7 in \cite{delyon_simulation_2006} to Riemannian manifolds. Define the process 
$\psi_t := \exp\left[-\frac{1}{2}\norm[\big]{\sigma^{-1}_t(\phi(z))u^{-1}\left(\frac{\nabla r_v(z)^2}{2(T-t)^{1/2}} \right)}^2\right]$.

\begin{lemma}\label{lemma conditional expectation}
Let $(D, \phi)$ be a normal neighborhood centered at $v$ and $0 < t_1 < t_2 <...<t_N < T$ and $g \in C^{\infty}_b(\M)$ (smooth bounded function with compact support in $\M$), then, with $\psi_t$ as defined above, we have
\begin{equation*}
\lim_{t\rightarrow T} \frac{\E[g(x_{t_1},...,x_{t_N})\psi_t]}{\E[\psi_t]} = \E[g(x_{t_1},...,x_{t_N})|X_T=v].
\end{equation*}
\end{lemma}

\begin{proof}
First, since the cut locus of any complete connected manifold has (volume) measure zero, we can integrate indifferently in any exponential chart. For any $t \in (t_N,T)$, we have
	\begin{equation*}
		\frac{\E[g(x_{t_1},...,x_{t_N})\psi_t]}{\E[\psi_t]}
		= 	
		\frac
		{
		\int_\M \Phi_g(t,z)e^{-\frac{1}{2}\norm[\big]{\sigma^{-1}(t,\phi(z))u^{-1}\left(\frac{\nabla r_v(z)^2}{2(T-t)^{1/2}} \right)}^2}d\Vol(z)
		}			
		{
		\int_\M \Phi_1(t,z)e^{-\frac{1}{2}\norm[\big]{\sigma^{-1}(t,\phi(z))u^{-1}\left(\frac{\nabla r_v(z)^2}{2(T-t)^{1/2}} \right)}^2}d\Vol(z)
		},
	\end{equation*}
where $d\Vol(z)$ denotes the volume measure,
\begin{align*}
		\Phi_g&(t,z) =
		 \int_\M \cdot \cdot \cdot 
         \int_\M 							
         g(z_1,...,z_N)p(0,u;t_1,z_1)\cdot \cdot \cdot p(t_N,z_N;t,z)
         d\Vol(z_1) \cdot \cdot \cdot d\Vol(z_N),
\end{align*}
and of course $\Phi_1(t,z) = p(0,x_0;t,z)$. We can write the expectation above as 
\begin{align*}
		\int_\M 
		&
		\Phi_g(t,z)
		e^{-\frac{1}{2}\norm[\big]{\sigma^{-1}(t,\phi(z))u^{-1}\left(\frac{\nabla r_v(z)^2}{2(T-t)^{1/2}} \right)}^2}d\Vol(z)
		=
		\\
		&
		\int_{\phi(\M)} \Phi_g(t,\phi^{-1}(x))e^{-\frac{1}{2}\norm[\big]{\sigma^{-1}(t,x)u^{-1}\left(\frac{\nabla r_v(\phi^{-1}(x))^2}{2(T-t)^{1/2}} \right)}^2}\sqrt{\det(G(\phi^{-1}(x)))}dx,
\end{align*}
where $G$ is the matrix representation of the Riemannian metric. As we are in a normal neighborhood, and $\{e_i\}$ is an orthonormal basis of $T_vM$, we have in particular $r_v(\phi^{-1}(x)) := d(\phi^{-1}(x),v) = \norm{\Log_v(\phi^{-1}(x))} = \norm{\Log_v \circ \Exp_v \circ E(x)}= \norm{E(x)} = \norm{x}$. Thus, if we apply the change of variable $x=(T-t)^{1/2}y$ we get 
\begin{align*}
	(T-t)^{-\frac{d}{2}} 
	&
	\int_{\phi(\M)} \Phi_g(t,\phi^{-1}(x))e^{-\frac{1}{2}\norm[\big]{\sigma^{-1}(t,x)u^{-1}\left(\frac{\nabla r_v(\phi^{-1}(x))^2}{2(T-t)^{1/2}} \right)}^2}\sqrt{\det(G(\phi^{-1}(x)))}dx 		
	\\	
	=&
	\int_{\phi(\M)} \Phi_g(t,\phi^{-1}((T-t)^{\frac{1}{2}}y))H\left(t,y\right)\sqrt{\det(G(\phi^{-1}((T-t)^{\frac{1}{2}}			y)))}dy 
	\\
	 \overset{t \rightarrow T}{\rightarrow}
	 &		\int_{\phi(\M)} \Phi_g(T,\phi^{-1}(0))H(T,y)							\sqrt{\det(G(\phi^{-1}(0)))}dy 
	\\
	=& 
	\Phi_g(T,v) \sqrt{\det(G(v))}\int_{\phi(\M)}e^{-\frac{1}{2}\norm[\big]{\sigma^{-1}(T,0)u^{-1}\left(\frac{\nabla r_v(\phi^{-1}(y))^2}{2} \right)}^2}dy,
\end{align*}
where 
$
	H(t,y) 
	=
	\exp\left[-\frac{1}{2}\norm[\big]{\sigma^{-1}\left(t,(T-t)^{\frac{1}{2}}y\right)u^{-1}\left(\frac{\nabla r_v(\phi^{-1}(y))^2}{2} \right)}^2\right].
$
Therefore, Bayes formula implies that
	\begin{equation*}
		\lim_{t\rightarrow T} \frac{\E[g(x_{t_1},...,x_{t_N})\psi_t]}			{\E[\psi_t]} = \frac{\Phi_g(T,v)}{\Phi_1(T,v)} = 						\E[g(x_{t_1},...,x_{t_N})|X_T=v].
	\end{equation*}
\end{proof}

We are now ready to complete the proof of \refthm{main result theorem}.
\begin{proof}[Proof \refthm{main result theorem}]
The Radon-Nikodym derivative \eqref{eq:radon-nikodym derivative for radial semimartingale} together with Novikov's condition ensures the equivalence of the measures of $X|X_T=v$ and $Y$ on $[0,t]$, for every $t < T$. By \reflemma{lemma:sde for Y} and \reflemma{lemma: g function}, we obtain the expressions for $\{\varphi_t\colon t \in [0,T)\}$ and $\{\psi_t\colon t \in [0,T)\}$. The proof is concluded by \reflemma{lemma conditional expectation}.
\end{proof}

Recall that in the case of $\cL_z = \tfrac{1}{2}\Delta_M$, the term $\varphi_t$ has the particular form $
	\varphi_t = \exp\left[\int_0^t \frac{r(Y_s)}{T-s}\left(d\eta_s + dL_s\right)\right],
$
with $d \eta_s = \tfrac{\partial}{\partial r}\log \Theta^{-\frac{1}{2}}ds$. We then have the following result of $L_1$-convergence of $\varphi$ as described in \cite{thompson2015submanifold}.

\begin{lemma}\label{lemma: L1 convergence}
	With $\varphi_t$ as defined above then $\varphi_t \overset{L_1}{\rightarrow} \varphi_T$.
\end{lemma}
\begin{proof}
	Note that for each $t \in [0,T)$ we have $\E^{\Qb}[\varphi_t]<\infty$ as well as $\varphi_t \overset{}{\rightarrow} \varphi_T$ almost surely by \refprop{prop:almost sure convergence}. The result then follows from the uniform integrability of $\left\{\varphi_t \colon t \in [0,T)\right\}$, which can be found in Appendix C.2 in \cite{thompson2015submanifold}.
\end{proof}

\begin{proof}[Proof \refthm{main result theorem for BM}]
When the driving semimartingale is a standard Brownian motion, we recall that the Radon-Nikodym derivative is given by \eqref{eq: radon-nykodym derivative for BM}.
In this context the function $g$, defined by \eqref{eq:g function}, reduces to the expression (suppressing the $v$-dependence)
	\begin{align*}
		g(t,X)= \frac{r_v(X_t)^2}{T-t} := \frac{r_t^2}{T-t}.
	\end{align*} 
The geometric It\^o's formula applied to this function then yields, coming from \eqref{radial representation by kendall}, the SDE
\begin{align}\label{ito formula for drift term in BM case}
    &-\frac{1}{2}\int_0^t \frac{r_s^2}{(T-s)^2}ds 
	= 
    \\&\quad
	- 
	\frac{1}{2}\frac{r_t^2}{T-t} + \int_0^t \frac{r_s}{T-s}d \beta_s + \frac{1}{2}\int_0^t \frac{r_s \Delta_\M r_s}{T-s}ds
	 + 
	 \frac{1}{2}\int_0^t  \frac{1}{T-s} ds - \int_0^t \frac{r_s}{T-s}dL_s.
     \nonumber
\end{align}
We see that \eqref{ito formula for drift term in BM case} substituted into \eqref{eq: radon-nykodym derivative for BM} yields, for any bounded $\mathcal{B}_t(W(\M))$-measurable $F$,
\begin{align*}
	&\mathbb{E}^{\mathbb{Q}}[F(Y)] =  C_t \mathbb{E}^{\mathbb{P}}\bigg[F(X)\exp\bigg[- 					\frac{r_t^2}{2(T-t)} + \frac{1}{2}\int_0^t \frac{r_s					\Delta r_s}{T-s}ds -  \int_0^t \frac{r_s}{T-s}dL_s\bigg]				\bigg],
\end{align*}
$\mathbb{Q}_t \text{ - a.s.}$, where we used that $C_t = \left(T/(T-t)\right)^{1/2}$. We can equivalently write the above as
\begin{equation}\label{Equal in expectation}
	\mathbb{E}^{\mathbb{Q}}[F(Y)\varphi_t] = \tilde{C}_t\mathbb{E}^{\mathbb{P}}		[F(X)\psi_t],
\end{equation}
where
	\begin{equation*}\label{phi and psi processes}
		\varphi_t = 
		\exp\bigg[ \int_0^t  \frac{r_v(Y_s)}{T-s} \left(d\eta_s + dL_s\right)\bigg],	\qquad
		\psi_t = \exp\bigg[- \frac{r_v(X_t)^2}{2(T-t)}\bigg],
	\end{equation*}
	and where $\tilde{C}_t = \left(T/(T-t)\right)^{d/2}$.
    This follows from
	\begin{equation*}
		\Delta r_v = \frac{d-1}{r_t} + \frac{\partial}{\partial r_t} \log \Theta_v,
	\end{equation*}
which holds on $\M \backslash \Cut(v)$, where $\Theta$ is the Jacobian determinant of the exponential map. Therefore we can rewrite $D_t$ as
	\begin{equation*}
	D_t = 
		\exp\bigg[ -\frac{r_v(X_t)^2}{2(T-t)} + \frac{(d-1)}{2}\int_0^t \frac{1}{T-s}ds +\int_0^t  \frac{r_v(X_s)}{T-s} \left(d\eta + dL_s\right)\bigg],
	\end{equation*}
where 
$
	d\eta = \frac{\partial}{\partial r} \log \Theta^{-\frac{1}{2}}ds.
$
Letting $F \equiv 1$ in equation \eqref{Equal in expectation} we obtain
	\begin{equation*}
	\frac{\mathbb{E}^{\mathbb{Q}}[F(Y)\varphi_t]}{\mathbb{E}^{\mathbb{Q}}[\varphi_t]}= \frac{\mathbb{E}^{\mathbb{P}}[F(X)\psi_t]}		{\mathbb{E}^{\mathbb{P}}[\psi_t]}.
	\end{equation*}
From \reflemma{lemma conditional expectation} and \reflemma{lemma: L1 convergence} we have, as $t \uparrow T$,
	\begin{equation*}
	\frac{\mathbb{E}^{\mathbb{Q}}[F(Y)\varphi_t]}{\mathbb{E}^{\mathbb{Q}}[\varphi_t]} \rightarrow 
	C \mathbb{E}^{\mathbb{Q}}[F(Y)\varphi_T] \qquad
	 \frac{\mathbb{E}^{\mathbb{P}}[F(X)\psi_t]}{\mathbb{E}^{\mathbb{P}}[\psi_t]} \overset{}{\rightarrow} \mathbb{E}^{\mathbb{P}}[F(X)|X_T=v],
	\end{equation*}
which concludes the proof.
\end{proof}

\section{Numerical experiments}\label{sec:numerical experiments}
In this section, we illustrate our simulation scheme on certain types of manifolds, including examples of two- and three-dimensional manifolds. Furthermore, the radial bridge simulation scheme is used for maximum likelihood estimation of diffusion mean values \cite{hansen2021diffusiongeometric}. The code used to generate the illustrations is available at \url{https://bitbucket.org/stefansommer/theanogeometry}.

\subsection{Simulation of bridge process on 2- and 3-dimensional manifolds}
We first let the manifold $\M$ be the $2$-torus $\M = \T^2$ with the standard embedding in $\mathbb R^3$. Our goal is to simulate a conditioned process on $\T^2$, where we condition on a point in the cut locus of the initial point. Figure \ref{fig: torus}(a) shows four sample paths of the simulated bridge with the initial value depicted by the red point. The process is conditioned to arrive at the black point, which is in the initial point's cut locus, at $T=1$. Figure \ref{fig: torus}(b) presents the underlying squared radial vector field. The radial bridge's drift term follows the radial vector field multiplied by an increasing time-dependent scalar.

\begin{figure}[H] 
    \centering
	\subfloat[Four sample paths from the simulation scheme of the radial bridge, $X_t$, from $x$ (red point) to $v$ (black point).]{\includegraphics[width=.4\linewidth]{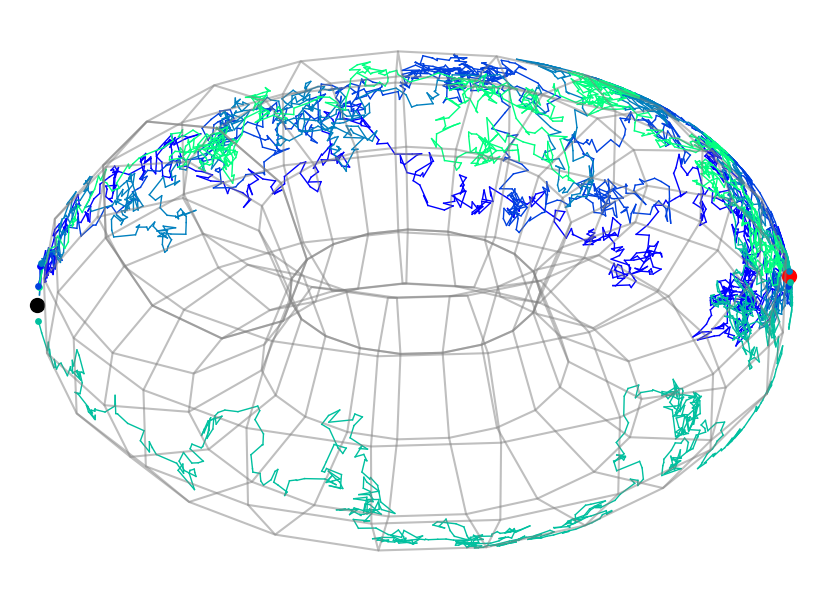}}
    \quad
	\subfloat[Radial vector field on the torus, related to the radial bridge, $X_t$, centered at the point $X_T=v$.]{\includegraphics[width=.4\linewidth]{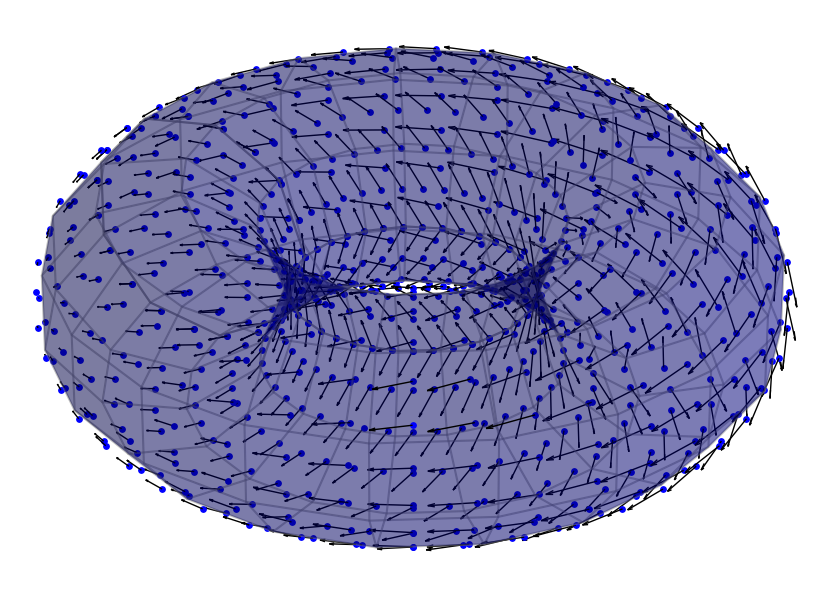}}
    \caption{}
    \label{fig: torus}
\end{figure}

In the second example, we consider another two-dimensional manifold, namely the cylinder $\M = \mathbb{S}^1 \times \R$. On the cylinder, the cut locus of a point $v = (p,x) \in \mathbb{S}^1 \times \R$ is the set $\Cut(v) = \{q\} \times \R$, where $q \in \mathbb{S}^1$ is antipodal to $p$. This is illustrated in Figure \ref{fig:cylinder cut locus}. 

\begin{figure}[H]
\centering
	\includegraphics[width=.29\linewidth]{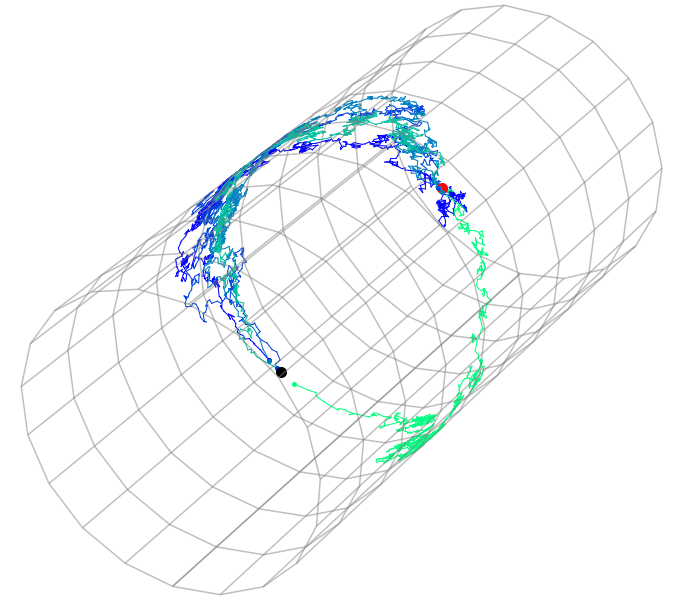}
	\caption{An example of four sample paths of the radial bridge conditioned to arrive at a point in the cut locus, $\Cut(x)$, of the initial point $x$.}\label{fig:cylinder bridges}
   \label{fig:cylinder bridge}
\end{figure}

Any winding around the cylinder trivially makes the process cross the cut locus. In Figure \ref{fig:cylinder bridges}, we illustrate the radial bridge's sample paths ending up at the cut locus at the terminal time winding different ways around the cylinder.

In the third example, we show how our simulation scheme works on a three-dimensional manifold. We consider the special orthogonal group $\SO(3)$ consisting of $3\times 3$ orthogonal matrices of determinant 1. Figure \ref{fig:SO3} shows sample paths of the radial bridge from $x \in SO(3)$ to $v \in SO(3)$, where $x=\mathrm{Id}_3$ (columns: red, blue, green vectors) and $v$ (black) is the terminal value $X_T=v$.

\begin{figure}[H] 
	\centering
	\subfloat{\includegraphics[width=.3\linewidth]{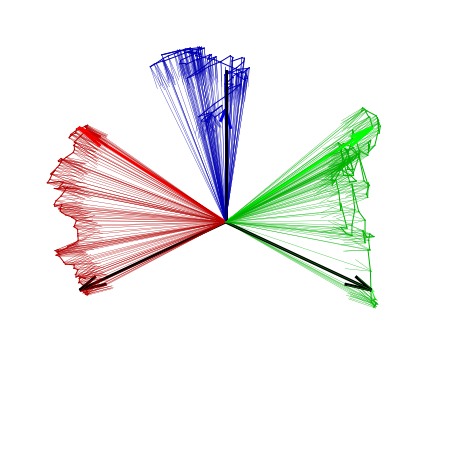}}
	\subfloat{\includegraphics[width=.3\linewidth]{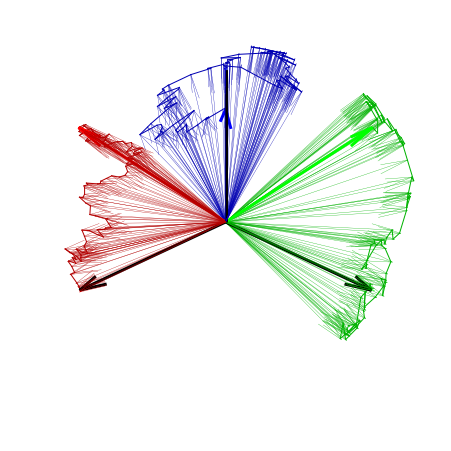}}
	\subfloat{\includegraphics[width=.3\linewidth]{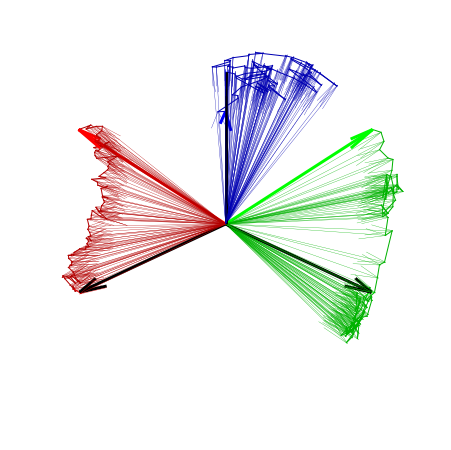}}
	\caption{The figure illustrates three sample paths from the radial bridge on $SO(3)$, by showing its left action on a basis of $\R^3$. The black arrows indicate the conditioned point.}\label{fig:SO3}
\end{figure}

\subsection{Bridge simulation for density estimation on 2-Sphere}
Bridge sampling can be used to estimate the transition density of a stochastic process. Assume that the family $\{\varphi_t\colon t \in [0,T)\}$ is uniformly integrable. As is seen from \reflemma{lemma conditional expectation}
	\begin{equation*}
	\E[C_t\psi_t] \overset{t \rightarrow T}{\longrightarrow} p(0,x_0;T,v) \sqrt{\det(G(v))}\int_{\phi(\M)}e^{-\frac{1}{2}\norm[\big]{\sigma^{-1}(T,0)u^{-1}\left(\frac{\nabla d(\phi^{-1}(y),v)^2}{2} \right)}^2}dy,
	\end{equation*}
which in the case of the $d$-sphere can be expressed as 
\begin{align*}
	\E[C_t\psi_t]
	&
	\overset{t \rightarrow T}{\longrightarrow} 
	p(0,x_0;T,v) \sqrt{\det(G(v))}\int_{\R^d}e^{-\frac{1}{2}y^TA(T,0)y}dy 
	\\
	&
	=
	\frac{p(0,x_0;T,v)\sqrt{(2\pi)^2}\sqrt{\det\left(G(v)\right)}}{\sqrt{\det\left(A(T,0)\right)}},
	\end{align*}
From the identity linking $\varphi_t $ to $\psi_t$, this leads to the following expression for the transition density with respect to the Riemannian volume form
	\begin{equation*}
	p(0,x_0;T,v) = \frac{\sqrt{\det\left(A(T,0)\right)}}{\sqrt{(2\pi T)^d}}e^{-\frac{1}{2}\norm[\big]{\sigma^{-1}(T,x_0)u^{-1}\left(\frac{\nabla d(x_0,v)^2}{2} \right)}^2}\E[\varphi_T],
	\end{equation*}
see also \cite{papaspiliopoulos2012importance}. In the Brownian motion setting, the above simplifies to	
	\begin{equation}
	p(0,x_0;T,v) = \frac{1}{\sqrt{(2\pi T)^d}}e^{-\frac{r_t^2}{2T}}\E[\varphi_T].
	\label{eq:p_approx_bridge}
	\end{equation}
Thus, we obtain an expression for the transition density of the Brownian motion.

On the sphere, we also have series expansions for the heat kernel as uniformly and absolute convergent power series
    \begin{equation}
        p(0,x;t,y) = \sum_{l=0}^{\infty} e^{-l(l+d-1)t}\frac{2l + d-1}{d-1} \frac{1}{A_{\mathbb S^{d}}}C_l^{\frac{d}{2}-1}(\langle x, y\rangle_{\R^{d+1}}),
	\label{eq:p_power}
    \end{equation}
for $x,y \in \mathbb S^d$, where $C_l^\alpha$ are the Gegenbauer polynomials and $A_{\mathbb S^{d}} = \frac{2\pi^{(d+1)/2}}{\Gamma((d+1)/2)}$ the surface area of $\mathbb S^d$ \cite{zhao2018exact}. In Figure \ref{fig:plot of densities}, we have plotted the transition density estimated from \eqref{eq:p_approx_bridge} using sampling to approximate $\E[\varphi_T]$ against a truncated version ($l=0,\ldots,16$) of \eqref{eq:p_power} as well as the density for the two-dimensional Euclidean Brownian motion. We have plotted the estimated densities along a geodesic running from the north pole to the south pole for three time points, $T=0.5, 1, 1.5$. 
	
\begin{figure}[H]
	\centering
	\center
	\includegraphics[width=.5\linewidth]{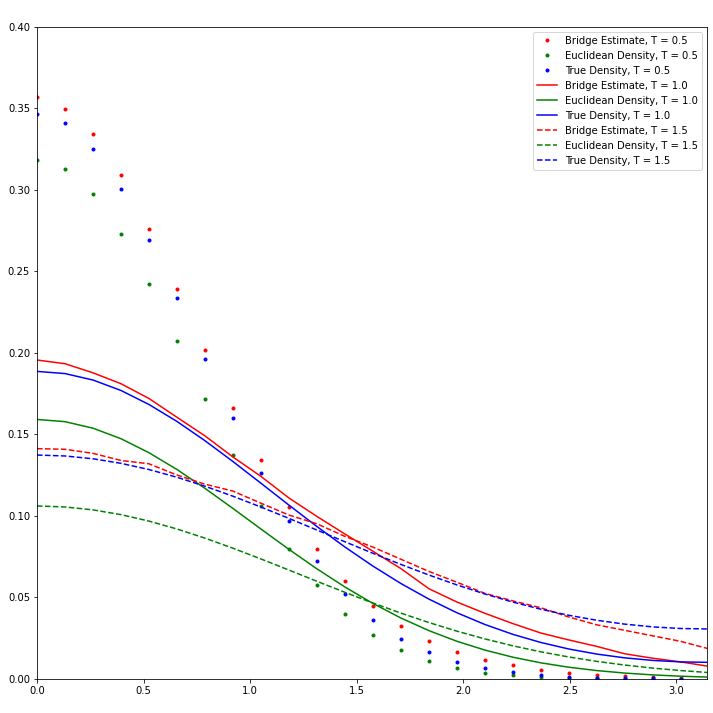}
	\caption{Estimated transition densities of a Brownian motion (red) on $\mathbb S^2$ using \eqref{eq:p_approx_bridge} with $T=0.5, 1, 2$. The densities are computed along a geodesic from the north-pole to the south-pole. The estimated densities are compared to an approximation using \eqref{eq:p_power} (blue) on $\mathbb S^2$ and densities of a $2$-dimensional Euclidean Brownian motion (green).}\label{fig:plot of densities}
\end{figure}
The resulting estimated heat kernel on $\mathbb S^2$, using the sampling scheme above, is displayed in \reffig{fig: sphere densities normalized time}. 
\begin{figure}[H]
	\centering
	\center
	\includegraphics[width=.29\linewidth,trim=150 100 300 100,clip=true]{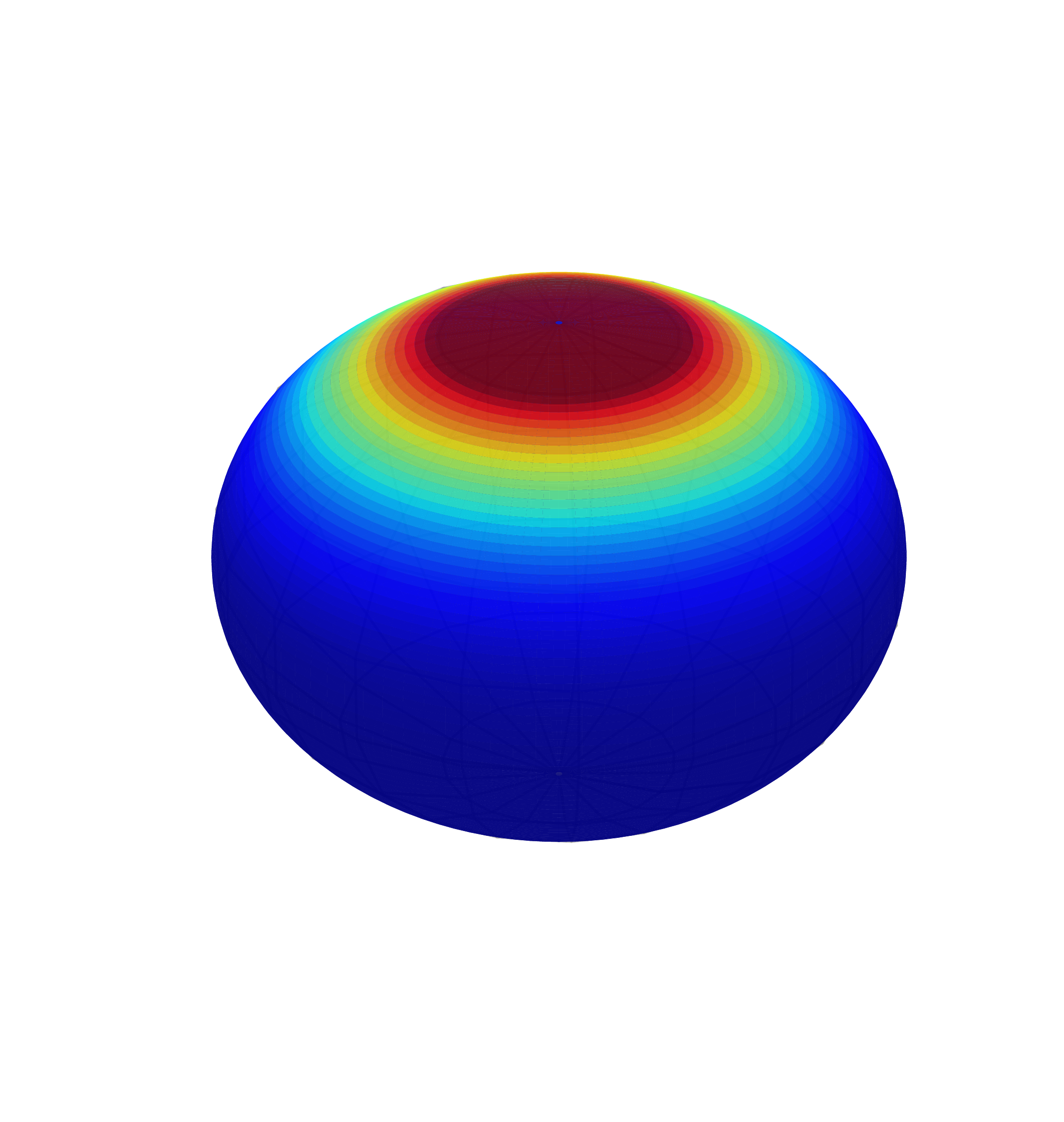}
	\includegraphics[width=.29\linewidth,trim=150 100 300 100,clip=true]{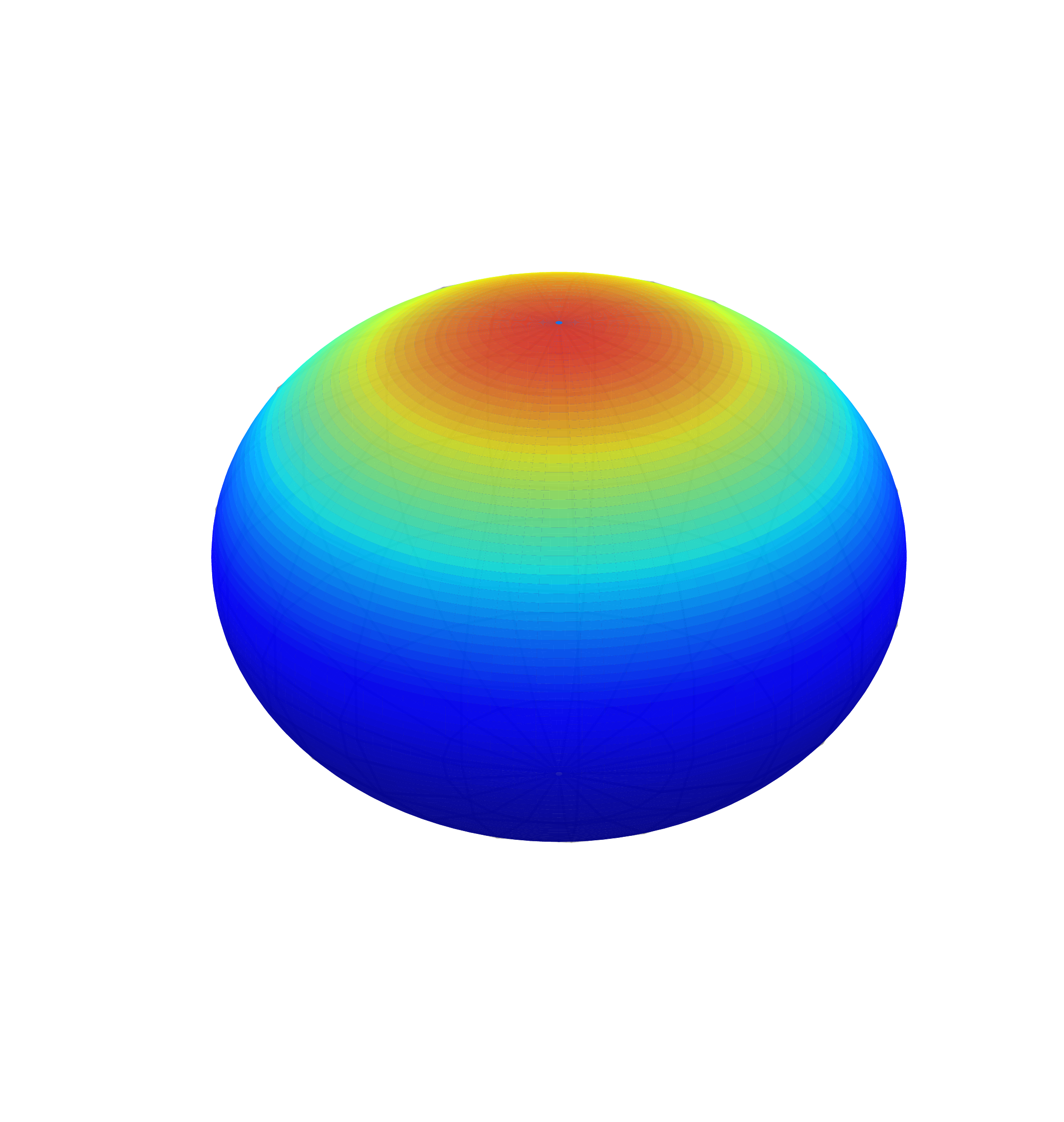}
	\includegraphics[width=.38\linewidth,trim=150 100 100 100,clip=true]{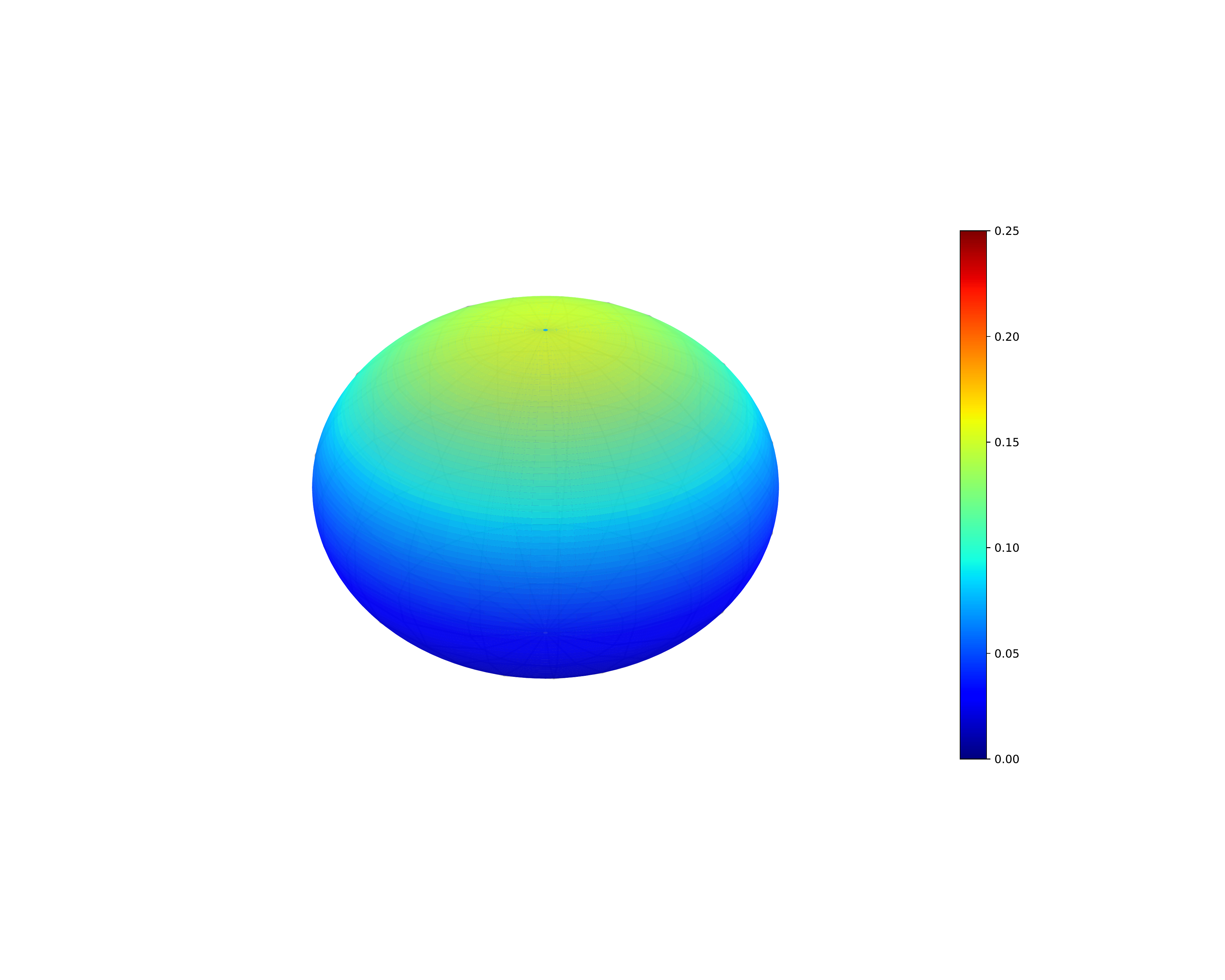}
	\caption{Estimated transition density on the $2$-sphere at times $T= t$, for $t=1/2, 3/4, 1$, respectively.}\label{fig: sphere densities normalized time}
\end{figure}

\reffig{fig: ellipsoid densities} shows the estimated heat kernel on a ellipsoid where no closed form solution is directly available.
\reffig{fig: torus densities} shows the estimated heat kernel on the torus $\T^2$.

\begin{figure}[H]
	\centering
	\center
	\includegraphics[width=.29\linewidth,trim=150 100 300 100,clip=true]{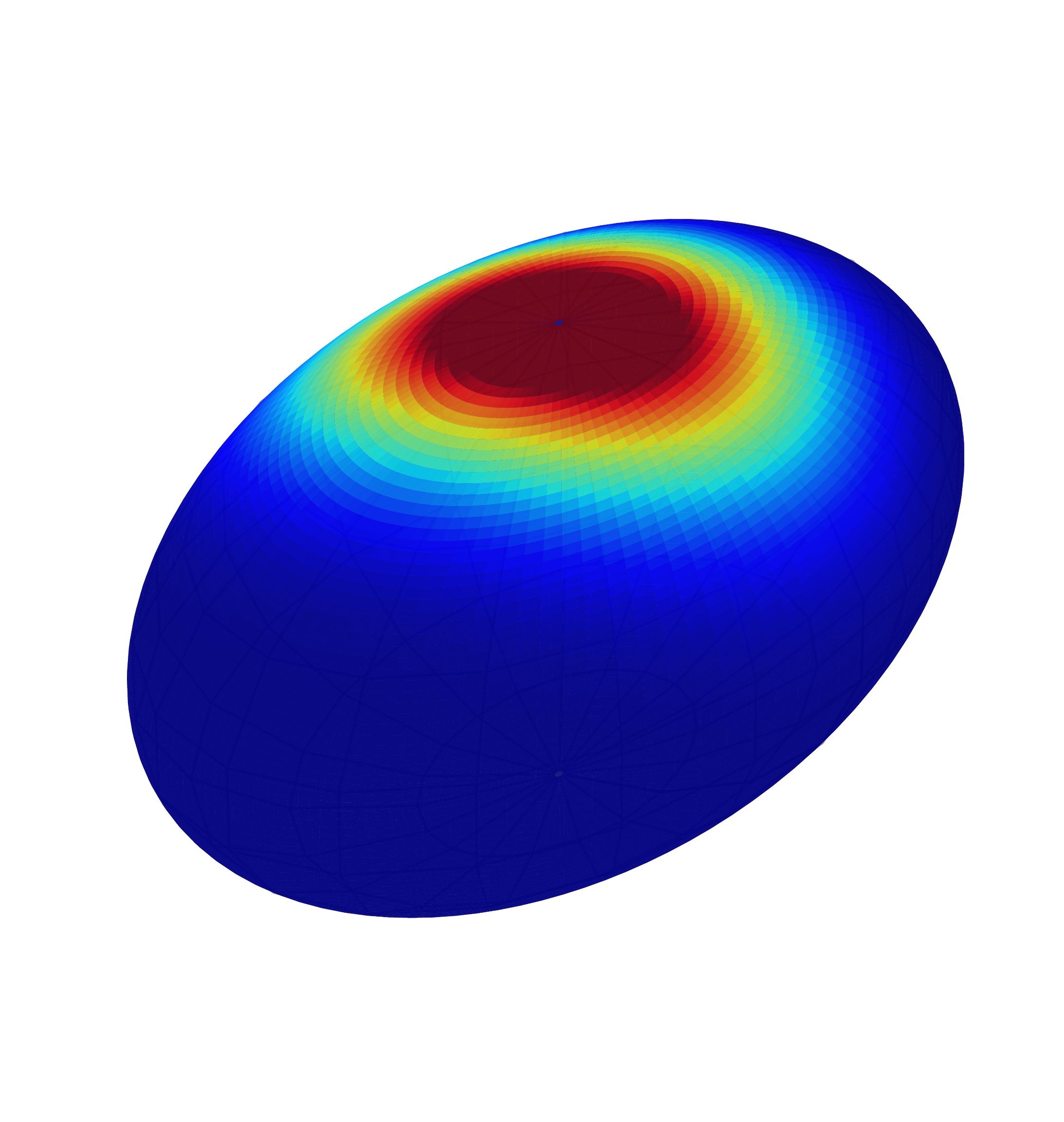}
	\includegraphics[width=.29\linewidth,trim=150 100 300 100,clip=true]{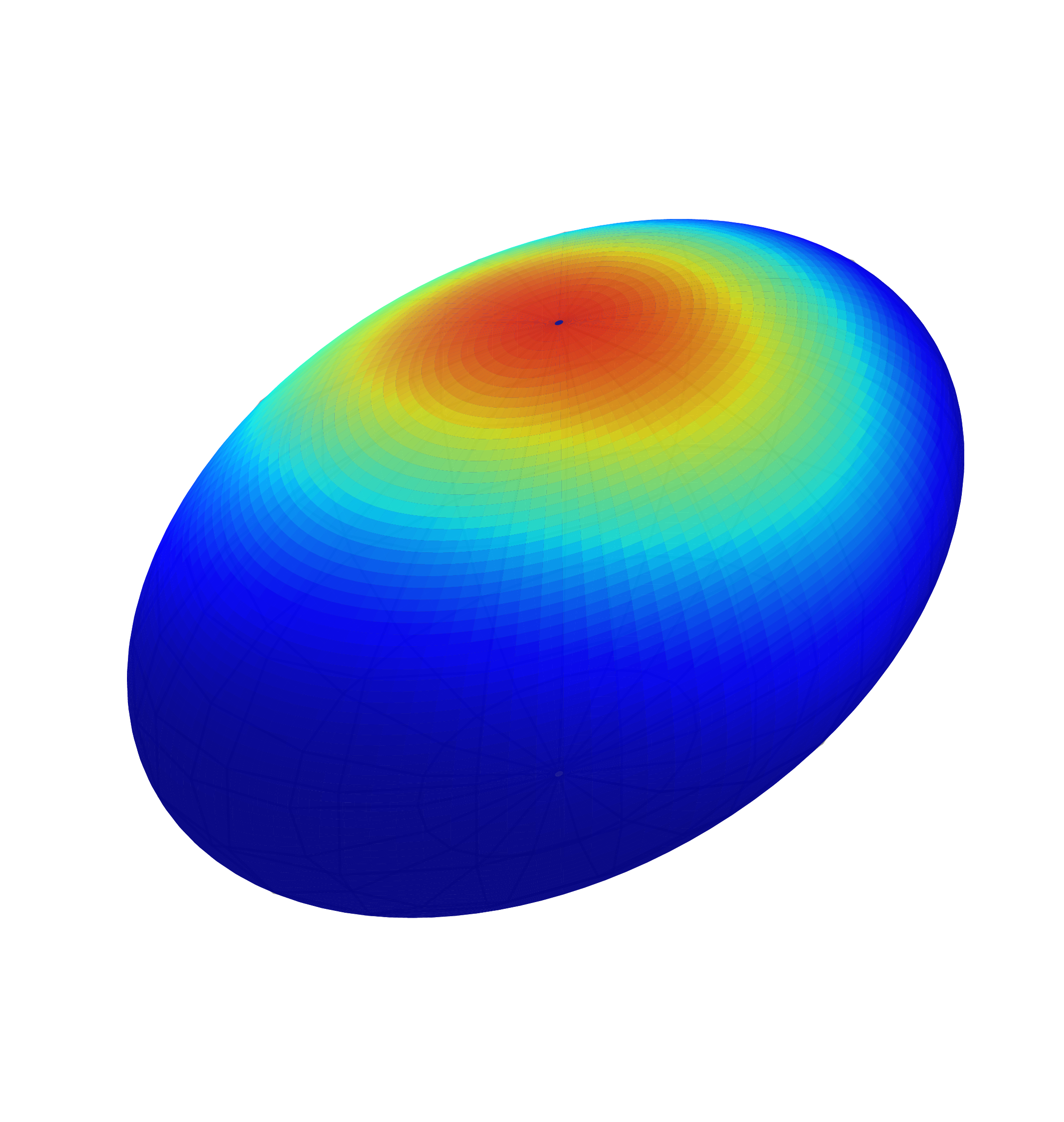}
	\includegraphics[width=.38\linewidth,trim=150 100 100 100,clip=true]{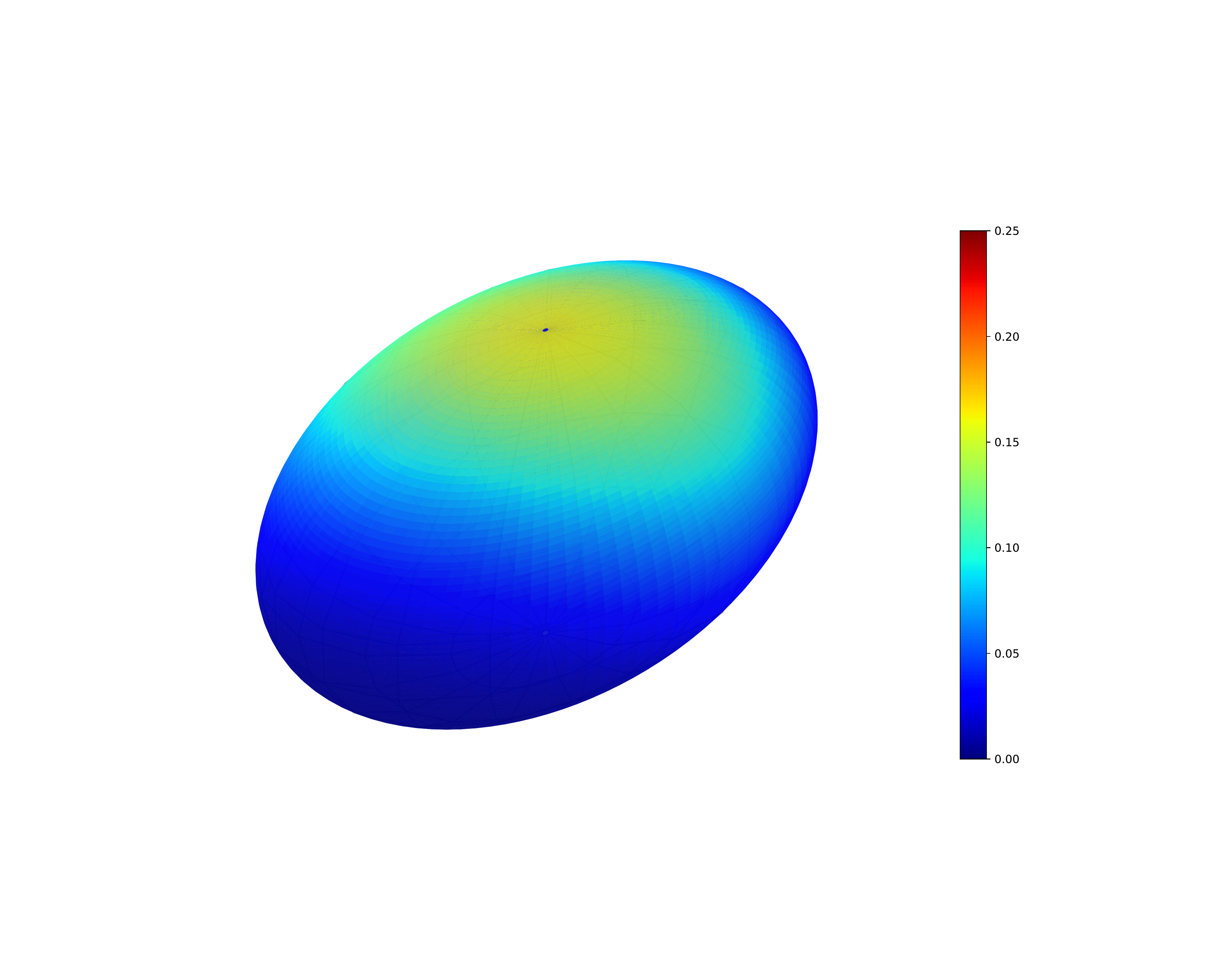}
	\caption{Estimated transition density on ellipsoid at times $T= t$, for $t=1/2, 3/4, 1$, respectively.}\label{fig: ellipsoid densities}
\end{figure}
	
\begin{figure}[H]
	\centering
	\center
	\includegraphics[width=.29\linewidth,trim=160 100 290 100,clip=true]{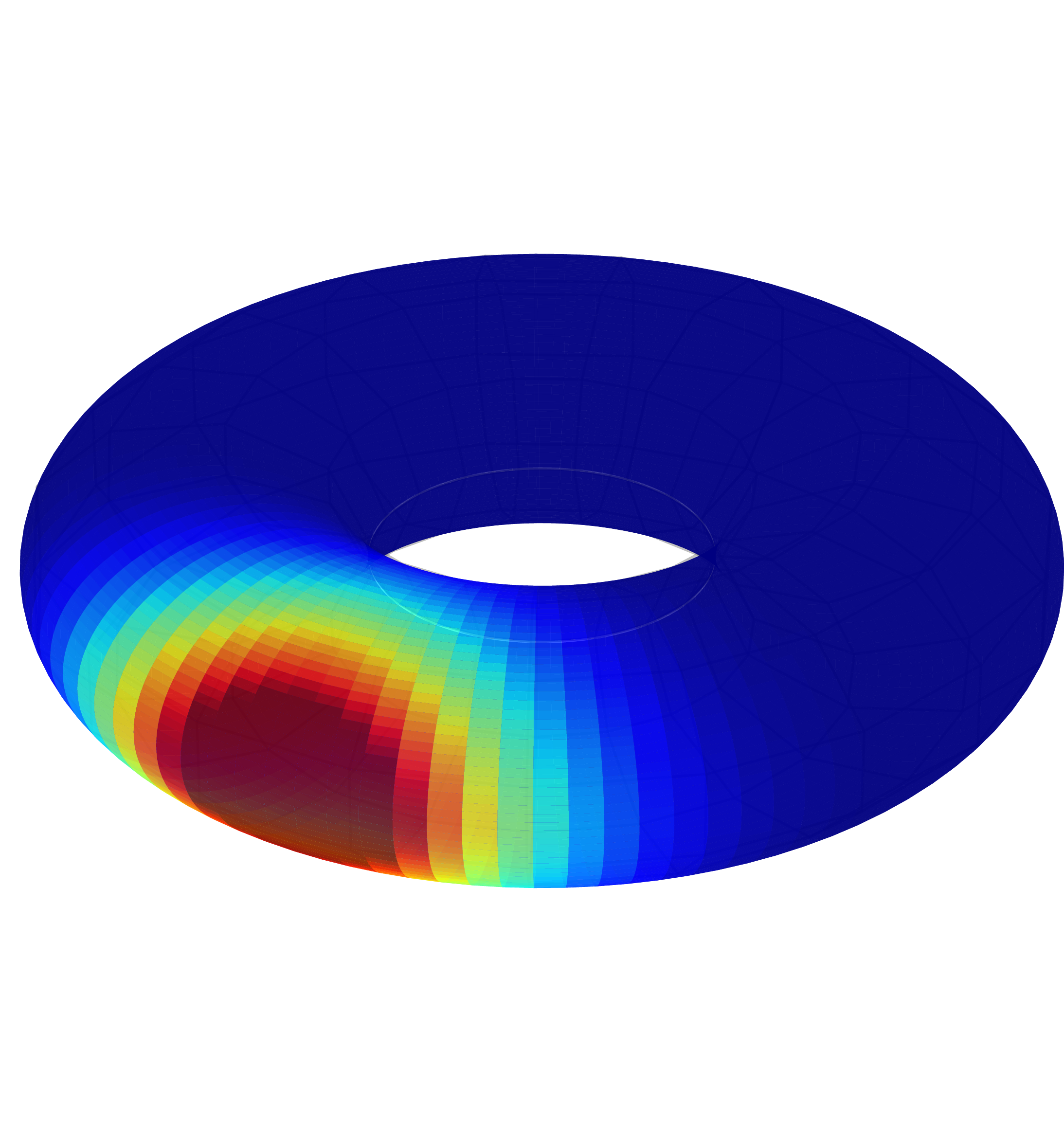}
	\includegraphics[width=.29\linewidth,trim=160 100 290 100,clip=true]{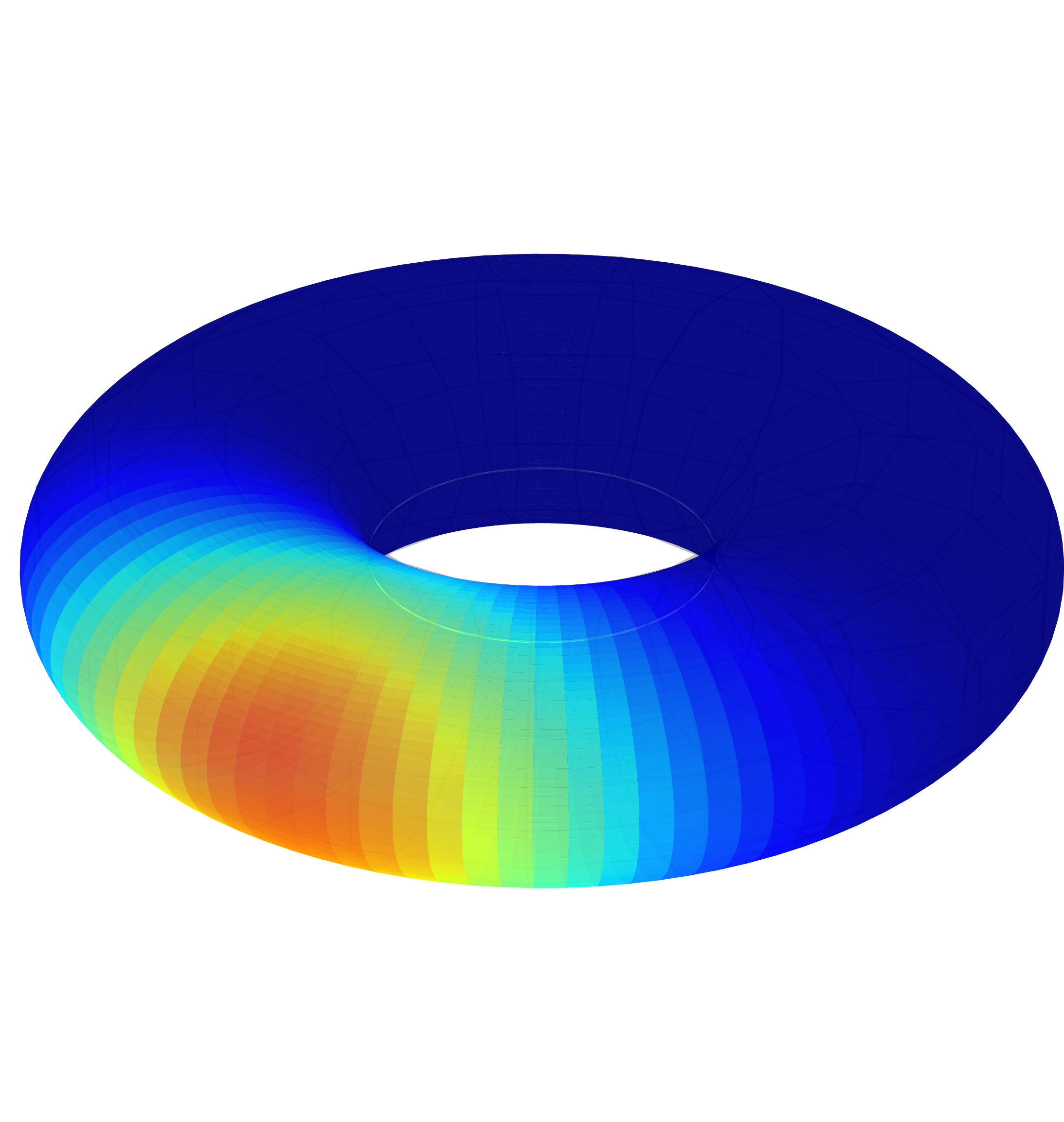}
	\includegraphics[width=.38\linewidth,trim=150 100 100 100,clip=true]{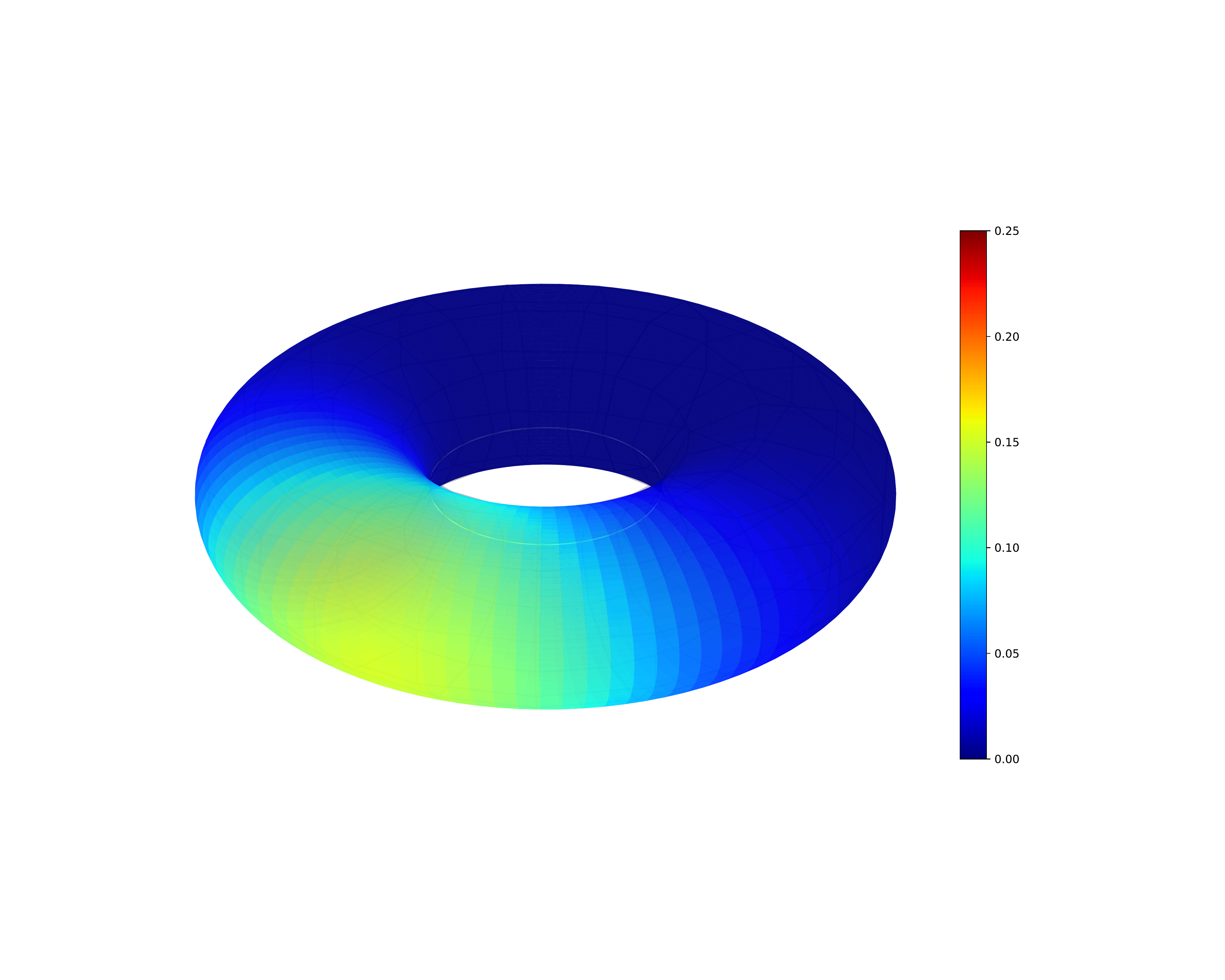}
	\caption{Estimated transition density on the embedded torus $\mathbb T^2$ at times $T= t$, for $t=1/2, 3/4, 1$, respectively.}\label{fig: torus densities}
\end{figure}

\subsection{Bridge simulation for diffusion mean estimation}\label{sec: example diffusion means}
We now demonstrate how the bridge sampling scheme can be used to estimate diffusion means \cite{hansen2021diffusiongeometric,hansen2021diffusion} on the two-sphere.
As described in the previous section, bridge sampling can be used to estimate the transition density of a Brownian motion. Using the approximation of the transition density \eqref{eq:p_approx_bridge} with sampling to approximate the expectation over $\varphi_T$, we can compute likelihood that make up the objective function for the diffusion mean. We can compute gradients of this quantity with respect to the initial valued and perform iterative optimization to find the diffusion mean. 

In figure \ref{fig: diffusion mean}(a), we sampled 100 points on $\mathbb S^2$ from a Brownian motion at $t=1$. Figure \ref{fig: diffusion mean}(b) illustrates how the initial guess of the mean (black dot) converges to the true mean (red dot) using an iterative optimization of the likelihood. From figure \ref{fig: diffusion mean}(c), we see how the iterated likelihood obtained converges. 
	\begin{figure}[ht] 
	\centering
	\subfloat[Sampled data points on $\mathbb{S}^2$.]{\includegraphics[width=.3\linewidth]{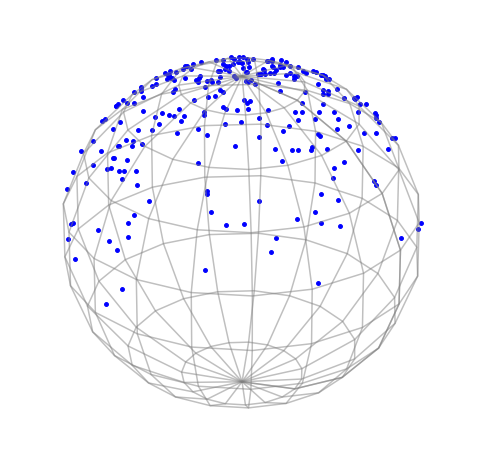}}
	\subfloat[Convergence to the diffusion mean (red dot) on $\mathbb{S}^2$ from initial guess (black dot).]{\includegraphics[width=.3\linewidth]{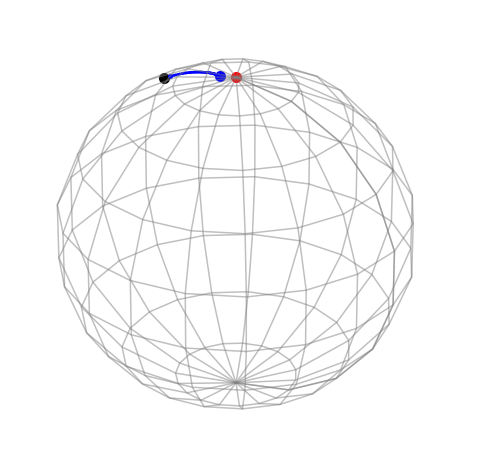}}
	\subfloat[The iterated likelihood.]{\includegraphics[width=.3\linewidth]{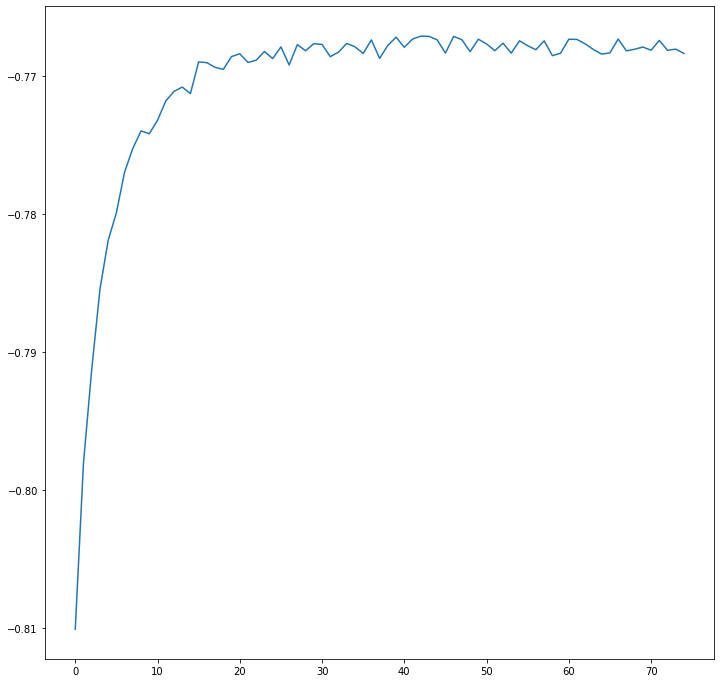}}
	\caption{The diffusion mean on $\mathbb S^2$ is estimated using our bridge sampling scheme. We sampled 100 points on $\mathbb S^2$ as endpoints of a Brownian motion started at the north pole. Using an initial guess of the mean, we sampled one bridge per observation to obtain an estimate for the likelihood and from this perform iterative optimization to maximize the likelihood. The iterative maximum likelihood approach yields a fast convergence after approximately 20 iterations.}\label{fig: diffusion mean}
	\end{figure}

\pagebreak
\section*{Acknowledgements}
The presented research was supported by the CSGB Centre for Stochastic Geometry and Advanced Bioimaging funded by a grant from the Villum Foundation.
S. Sommer is supported by the Villum Foundation Grants 40582 and the Novo Nordisk Foundation grant NNF18OC0052000.

\bibliographystyle{tfs}
\bibliography{bibfile}

\end{document}